\documentclass[pdflatex, 10.5pt]{amsart}
\usepackage{hyperref}
\usepackage{mathrsfs}
\usepackage{amssymb}
\usepackage{amsthm}
\usepackage{ascmac}
\usepackage{url}
\usepackage{xcolor}
\usepackage{mathtools}
\mathtoolsset{showonlyrefs=true}
\usepackage{tikz}
\usepackage{tikz-cd}
\usetikzlibrary{intersections, calc, arrows}
\usepackage{here}
\usepackage{Article}
\usepackage{Symbols}
\usepackage[
	backend=biber,
	style=numeric,
	sorting=nty,
	arxiv=abs,
	doi=false,
	url=false,
	isbn=false,
    eprint=false
]{biblatex}

\counterwithin{equation}{section}

\DeclareFieldFormat*{title}{\mkbibquote{#1\adddot}}
\DeclareFieldFormat[article]{book}{\mkbibemph{#1}}
\DeclareFieldFormat*{volume}{\mkbibbold{#1}}
\renewbibmacro{in:}{}
\DeclareFieldFormat{pages}{#1}
\renewbibmacro*{journal+issuetitle}{%
  \usebibmacro{journal}%
  \setunit*{\addspace}%
  \iffieldundef{series}
    {}
    {\newunit\printfield{series}%
     \setunit{\addspace}}%
  \usebibmacro{volume+number+eid}%
  \setunit{\addspace}%
  \usebibmacro{issue+date}%
  \setunit{\addcolon\space}%
  \usebibmacro{issue}%
  \newunit}
\DeclareFieldFormat*{title}{\mkbibemph{#1}}
\DeclareFieldFormat*{journaltitle}{\textnormal{#1}}
\title[the iterated limit of a quaternary of means]{The algebro-geometric aspect of the iterated limit of a quaternary of means of four terms}
\author{{\sc Matsumoto} Keiji}
\address{
	Department of Mathematics, Faculty of Science, Hokkaido University, Sapporo 060-0810, Japan
	}
	\email{matsu@math.sci.hokudai.ac.jp}

\author{{\sc Nakano} Ryunosuke}
\address{
	Graduate School of Science, Hokkaido University, Sapporo 060-0810, Japan
	}
	\email{nakano.ryunosuke.i3@elms.hokudai.ac.jp}

\subjclass[2020]{11F55 (Primary), 33C65 (Secondary)}
\keywords{Arithmetic-geometric mean, automorphic forms, theta constants, the Lauricella hypergeometric series.}
\thanks{The second named author is partially supported by JST SPRING, Grant Number JPMJSP2119.}

\begin{document}
\begin{abstract}
	We study the iterated limit of a quaternary of means of four terms
	through the period map from the family of cyclic fourfold coverings
	of the complex projective line branching at six points to the three-dimensional
	complex ball \(\mathbb{B}_3\) embedded into the Siegel upper half-space of
	degree four.
	We construct four automorphic forms on \(\mathbb{B}_3\) expressing the inverse of
	the period map, and give an equality between one of them and a period integral,
	which is an analogue of Jacobi's formula between
	a theta constant and an elliptic integral.
	We find a transformation of \(\mathbb{B}_3\) such that
	the quaternary of means appears through its action on the four automorphic forms.
	These results enable us to express the iterated limit by
	the Lauricella hypergeometric series of type \(D\) in three variables.
\end{abstract}
\maketitle


\pagenumbering{arabic}
\section{Introduction}
As a historical background of this research, we briefly review
the arithmetic-geometric mean and its analogues.
For real numbers \(0<b<a\), we give sequences \(\{a_n\}\) and \(\{b_n\}\)
with initial terms \(a_0=a\) and \(b_0=b\) by the recurrence relations
\[
	a_{n+1} = \frac{a_n + b_n}{2},\quad b_{n+1} = \sqrt{a_n b_n},
\]
induced from the arithmetic mean and the geometric mean.
It is easy to see that
\[
	\lim\limits_{n\to \infty}a_n=\lim\limits_{n\to \infty}b_n,
\]
which is called the Gauss AGM and denoted by \(M_\mathrm{G}(a,b)\).
C. F. Gauss showed that it is related to an elliptic integral and to
the hypergeometric series as
\begin{equation}\label{eq:AGM=F}
	\frac{a}{M_\mathrm{G}(a,b)}=
	\frac{a}{\pi}\int_{-\infty}^\infty \frac{dt}{\sqrt{(t^2+a^2)(t^2+b^2)}}=
	F\left(\frac{1}{2},\frac{1}{2},1;
	1-\frac{b^2}{a^2}\right).
\end{equation}
Here the hypergeometric series is defined by the power series
\[
	F(\alpha,\beta,\gamma;z) = \sum_{n=0}^\infty
	\frac{(\alpha,n)(\beta,n)}{(\gamma,n)(1,n)}z^n
	\quad
\]
in a variable \(z\) with complex parameters \(\alpha,\beta,\gamma(
\ne 0,-1,-2,\dots)\),  where
\((\alpha,n)  = \alpha(\alpha+1)\cdots(\alpha+n-1)=
{\varGamma(\alpha+n)}/{\varGamma(\alpha)}\).
It absolutely converges on the unit disk \(\mathbb{D}=\{z\in \C\mid |z|<1\}\).

On the other hand, the arithmetic mean and the geometric mean appear in
the duplication formulas
\begin{equation}\label{eq:2tau formula}
	\vartheta_{00}(2\tau_1)^2=\frac{\vartheta_{00}(\tau_1)^2+\vartheta_{01}(\tau_1)^2}{2},\quad
	\vartheta_{01}(2\tau_1)^2=\vartheta_{00}(\tau_1)\vartheta_{01}(\tau_1),
\end{equation}
of Jacobi's theta constants
\[
	\vartheta_{jk}(\tau_1) = \sum_{n\in \mathbb{Z}} \e 	\left(	\frac{1}{2}\left(n+\frac{j}{2}\right)^2\tau_1+	\frac{k}{2}
	\left(n+\frac{j}{2}\right)\right)
	\quad (j,k\in \{0,1\})
\]
defined on the upper half-plane
\(\mathbb{H}=\{\tau_1\in \mathbb{C}\mid \Im(\tau_1)>0\}\),
where \(i\) denotes the imaginary unit \(\sqrt{-1}\),
and \(\e(t)=\exp(2\pi i t)\).
We can show \eqref{eq:AGM=F} by
the properties \(\lim\limits_{n\to \infty }\vartheta_{00}(2^n\tau_1)^2=
\lim\limits_{n\to \infty }\vartheta_{01}(2^n\tau_1)^2=1\) and
\(z =\frac{\vartheta_{10}(\tau_1(z ))^4}
{\vartheta_{00}(\tau_1(z ))^4}\) for \(\tau_1(z )=i\frac{F\left(\frac{1}{2},\frac{1}{2},1;1-z \right)}{F\left(\frac{1}{2},\frac{1}{2},1;z \right)}\),
and Jacobi's formula
\begin{equation}\label{eq:Jacobi-formula}
	F\left(\frac{1}{2},\frac{1}{2},1;z \right)=\vartheta_{00}(\tau_1(z ))^2.
\end{equation}


In 1991, J. M. Borwein and P. B. Borwein gave in \cite{Bor91} two sequences
with initial terms  \(a_0=a\), \(b_0=b\) \((0 < b \leq a)\)
by the recurrence relations
\[
	a_{n+1} = \frac{a_n+2b_n}{3},\quad b_{n+1} = \sqrt[3]{b_n \frac{a_n^2 + a_n b_n + b_n^2}{3}},
\]
induced from two generalized means.
Then the two sequences converge to a common limit called the Borwein cubic AGM.
By considering the \(3\tau\) formula for theta constants
with respect to the \(A_2\)-lattice and Jacobi's formula between
a theta constant of this kind and \(F(\frac{1}{3}, \frac{2}{3}, 1;z^3)\),
they expressed this AGM by \(a/F(\frac{1}{3}, \frac{2}{3}, 1;1-\frac{b^3}{a^3})\).
In \cite{KS07}, K. Koike and H. Shiga extended the two sequences
to three sequences by the recurrence relations induced from three means of
three terms, and studied them by considering
the period map for the family of Picard's curves
of genus \(3\) and its inverse.
The common limit of the three sequences is expressed by
the Lauricella hypergeometric series \(F_D\left(\frac{1}{3},\frac{1}{3},\frac{1}{3},1;z_1,z_2\right)\)
of type \(D\) in two variables defined in \eqref{eq:HGS-FD}.

J. M. Borwein and P. B. Borwein  also introduced
two other sequences in \cite{Bor91} by the recurrence relations
\[
	a_{n+1} = \frac{a_n + 3b_n}{4},\quad
	b_{n+1} = \sqrt{\frac{a_n + b_n}{2}\,b_n}.
\]
They gave a formula for the common limit \(M_\mathrm{Bor}(a,b)\)
of these sequences as
\begin{equation}\label{prop:Borwein result 1}
	\frac{a}{M_\mathrm{Bor}(a,b)}
	= F\left( \frac{1}{4},\frac{3}{4},1;1-\frac{b^2}{a^2}
	\right)^2.
\end{equation}
To obtain this formula, they used the formula
\begin{equation}\label{prop:Borwein result 2}
	F\left( \frac{1}{4},\frac{3}{4},1;1-\frac{\beta(\tau_1)^2}{\alpha(\tau_1)^2}\right)^2
	= \alpha(\tau_1),
\end{equation}
where  \(\tau_1\in\mathbb{H}\) and
\[
	\alpha(\tau_1)=\vartheta_{00}(\tau_1)^4+\vartheta_{10}(\tau_1)^4,\quad
	\beta(\tau_1)=\vartheta_{00}(\tau_1)^4-\vartheta_{10}(\tau_1)^4.
\]

In 2009, T. Kato and K. Matsumoto established
the following result in \cite{KM09}.
Define four sequences \(\{a_n\}\), \(\{b_n\}\), \(\{c_n\}\), \(\{d_n\}\)
with initial terms \((a_0,b_0,c_0,d_0)=(a,b,c,d)\) \((0 < d \leq c \leq b \leq a)\)
by the recurrence relations
\begin{equation}\label{eq:means-KM09}
	\begin{array}{ll}
		a_{n+1}  = \dfrac{a_n + b_n + c_n + d_n}{4},\quad
		& b_{n+1}                                     = \dfrac{\sqrt{(a_n + d_n)(b_n + c_n)}}{2}, \\[
		1em]
		c_{n+1}  = \dfrac{\sqrt{(a_n + c_n)(b_n + d_n)}}{2},\quad
		& d_{n+1}                                     = \dfrac{\sqrt{(a_n + b_n)(c_n + d_n)}}{2}.
	\end{array}
\end{equation}
Then the four sequences converge to a common limit \(M_{\mathrm{KM}}(a,b,c,d)\),
and we have
\begin{equation}\label{eq:Matsumoto AGM}
	\frac{a}{M_{\mathrm{KM}}(a,b,c,d)}
	= F_D\left(\frac{1}{4} ,\frac{1}{4},\frac{1}{4},\frac{1}{4}, 1;1-\frac{b^2}{a^2},1-\frac{c^2}{a^2},1-\frac{d^2}{a^2}\right)^2,
\end{equation}
where \(F_D(\alpha,\beta_1,\beta_2,\beta_3,\gamma;z_1,z_2,z_3)\) is
the Lauricella hypergeometric series of type \(D\) in three variables
defined in \eqref{eq:HGS-FD}.
In the special case \(b=c=d\), it reduces to the Borwein AGM;
thus it can be regarded as an extension of the Borwein AGM.
Their proof of the equality \eqref{eq:Matsumoto AGM} is based on
the multidimensional version of
\cite[Theorem 8.3(Invariance Principle)]{Bor98}.
They showed that the right-hand side of \eqref{eq:Matsumoto AGM} satisfies
an invariance property with respect to the recurrence relations by the hypergeometric system of differential equations for \(F_D(\frac{1}{4},\frac{1}{4},\frac{1}{4},\frac{1}{4},1;z_1,z_2,z_3)\).

In this paper,  following \cite{KS07},
we study the formula \eqref{eq:Matsumoto AGM}
algebro-geometrically through the period map for the family of algebraic curves
\[
	C(x) \colon w^4 = z(z-x_1)(z-x_2)(z-x_3)(z-1)
\]
parameterized  by \(x = (x_1,x_2,x_3)\) in the set
\[
	X = \{(x_1,x_2,x_3) \in \mathbb{C}^3 \mid x_j \neq 0,1\ (j=1,2,3),\ x_j \neq x_k\ (1 \leq j < k \leq 3)\}.
\]
As shown in \cite{DM89} and \cite{Yos97},
the image of the period map is an open dense subset of
the \(3\)-dimensional complex ball \(\mathbb{B}_3\).
It is embedded into the Siegel upper half-space \(\mathfrak{S}_4\)
of degree \(4\). We construct automorphic forms on \(\mathbb{B}_3\)
by  using the pullback of Riemann's theta constant \(\vartheta_{a,b}(\tau)\) \((a,b\in \Z^4)\)
on \(\mathfrak{S}_4\) given in Definition \ref{def:Rieman-theta}
under this embedding.
We find four automorphic forms \(a(v)\), \(b_1(v)\), \(b_2(v)\), \(b_3(v)\)
on \(\mathbb{B}_3\) and a mean generating transformation
\(R:\mathbb{B}_3\ni v\mapsto Rv \in \mathbb{B}_3\)  such that the map
\[
	v \mapsto \big(1-\frac{b_1(v)^2}{a(v)^2},1-\frac{b_2(v)^2}{a(v)^2},
	1-\frac{b_3(v)^2}{a(v)^2}\big)
\]
coincides with the inverse of the period map and that
\begin{equation}\label{eq:Intro-means}
	\begin{array}{ll}
		a(Rv)    = \dfrac{a(v)\!+\!b_1(v)\!+\!b_2(v)\!+\!b_3(v)}{4},
		& b_1(R v) = \dfrac{\sqrt{(a(v)\!+\!b_3(v))(b_1(v)\!+\!b_2(v))}}{2},   \\ [2mm]
		b_2(R v) = \dfrac{\sqrt{(a(v)\!+\!b_2(v))(b_1(v)\!+\!b_3(v))}}{2}, & b_3(R v) = \dfrac{\sqrt{{(a(v)\!+\!b_1(v))(b_2(v)\!+\!b_3(v))}}}{2},
	\end{array}
\end{equation}
for \(v\) in the image \(\mathbb{B}_3^{123}\) of
\[
	X_{\mathbb{R}}^{123}=
	\{(x_1,x_2,x_3)\in \R^3\mid 0<x_1<x_2<x_3<1\}\subset X
\]
under the period map. The explicit form of \(R\) is given in
Definition \ref{def:MGT-R}, and
\(R\) does not induce the map
\(\mathfrak{S}_4 \ni \tau \mapsto 2\tau \in \mathfrak{S}_4\) through
the embedding \(\mathbb{B}_3\to \mathfrak{S}_4\).
We relate \(a(v)\) to \(F_D\left( \frac{1}{4},\frac{1}{4},\frac{1}{4},\frac{1}{4},1;x_1,x_2,x_3\right)\)
as \eqref{eq:Jacobi formula 3} in Theorem \ref{thm:Jacobi formula} (see also Definition \ref{def:abc-functions}),
which is regarded as a generalization of the formula
\eqref{prop:Borwein result 2}.
We  show the formula \eqref{eq:Matsumoto AGM} by these results.

We explain key ideas in this paper.
We regard the curve \(C(x)\) as a fourfold covering of the complex projective
line \(\mathbb{P}^1\) branching at \(0,x_1,x_2,x_3,1,\infty\) by
the natural projection
\(\pr:C(x)\ni (z,w) \mapsto z\in \mathbb{P}^1\)
with a covering transformation
\(\rho\colon C(x)\ni (z,w)\mapsto (z,iw) \in C(x)\).
Since this curve is of genus \(6\), its Jacobi variety is too large
to construct the period map to the \(3\)-dimensional complex ball.
Thus we consider the Prym variety
\(\Prym(C(x), \rho^2)=H^0_-(C(x),\Omega^1)^\ast / H_1^-(C(x),\Z)\) of \(C(x)\)
with respect to the involution \(\rho^2\),
where \(H^0_-(C(x),\Omega^1)\) and \(H_1^-(C(x),\Z)\)
are the \((-1)\)-eigenspaces of \(\rho^2\)
in the spaces \(H^0(C(x),\Omega^1)\) and \(H_1(C(x),\Z)\),  respectively,
and \(H^0_-(C(x),\Omega^1)^\ast\) is the dual space of \(H^0_-(C(x),\Omega^1)\).
Since the polarization of \(\Prym(C(x), \rho^2)\) is \((1,1,2,2)\),
we introduce a sublattice \(\Lambda=\Lambda(x)=\langle A_1,\ldots,A_4,B_1,\ldots,B_4 \rangle_\mathbb{Z}\) of \(H_1^-(C(x),\mathbb{Z})\) such that the quotient
\[
	A_\Lambda = H^0_-(C(x),\Omega^1)^\ast / \Lambda(x)
\]
is an Abelian variety with a principal polarization \((2,2,2,2)\), and that
the representation matrix of \(\rho\) with respect to the basis
\(\Transpose{(A_1,\ldots,A_4,B_1,\ldots,B_4)}\) becomes
\[
	\begin{pmatrix}
		O_4 & -U  \\
		U   & O_4
	\end{pmatrix},\quad
	U =
	\begin{pmatrix}
		0 & 1 & 0 & 0 \\
		1 & 0 & 0 & 0 \\
		0 & 0 & 1 & 0 \\
		0 & 0 & 0 & 1
	\end{pmatrix}.
\]
The lattice \(\Lambda(x)\) enables us to express
the period map \(\widetilde{\per}\)
from the universal covering of \(X\) to the \(3\)-dimensional complex ball
represented by
\[
	\mathbb{B}_3 = \{\xi \in \mathbb{P}^3 \mid \xi^\ast U \xi < 0\}
\]
with an embedding \(\imath\) from  \(\mathbb{B}_3\)
into the Siegel upper half-space \(\mathfrak{S}_4\) of degree \(4\),
where \(v^\ast = \Transpose{\bar{v}}\).
We also have the period map \(\per\) from \(X\) to
the quotient space \(\Gamma\backslash \mathbb{B}_3\) of \(\mathbb{B}_3\) by the monodromy group \(\Gamma\) of \(\widetilde{\per}\).

We construct \(\per^{-1}\) by considering the Abel-Jacobi-\(\Lambda\) map
\[
	\psi_\Lambda \colon C(x)\ni P \to (1-\rho^2)\int_{P_\infty}^{P}(\varphi_1,\dots,
	\varphi_4)\in  A_\Lambda=\C^4/\Lambda(x),
\]
where \(P_\infty=\pr^{-1}(\infty)\in C(x)\) and
\(\varphi_1,\dots,\varphi_4\) are in \(H^0_-(C(x),\Omega^1)\)
satisfying \[
	\int_{B_j} \varphi_k=\delta_{j,k}=\left\{
	\begin{array}{ccc} 1 & \textrm{if} & j=k,    \\
		0       & \textrm{if} & j\ne k, \\
	\end{array}
	\right.
\]
for \(B_1,\dots,B_4\in \Lambda(x)\subset H_1^-(C(x),\Z)\).
The Abel-Jacobi-\(\Lambda\) map \(\psi_\Lambda\) is well defined since
\((1-\rho^2)H_1(C(x),\mathbb{Z}) \subset \Lambda(x) \subset
H_1^-(C(x),\mathbb{Z})\).
We can see the order of zero
of the pullback \(\psi_\Lambda^\ast (\vartheta_{a,b}(\zeta,\tau))\)
of Riemann's theta function \(\vartheta_{a,b}(\zeta,\tau)\) given in
\eqref{eq:Rieman-theta}
under \(\psi_\Lambda\)
at the six ramification points
\(P_j=\pr^{-1}(j)\in C(x)\) \((j=0,x_1,x_2,x_3,1,\infty)\) by using the action of the covering transformation \(\rho\).
We have meromorphic functions on \(C(x)\) as ratios of them.
By using special values of these meromorphic functions, we have relations
among theta constants and \(x_1,x_2,x_3\).
In fact, \(x_1,x_2,x_3\) are expressed as
\begin{gather}
	x_1  = \frac{4\vartheta_0(v)^2 \vartheta_1(v)^2}{(\vartheta_0(v)^2 +\vartheta_1(v)^2)^2},\quad  x_2  = \frac{4\vartheta_4(v)^2 \vartheta_5(v)^2}{(\vartheta_4(v)^2 +\vartheta_5(v)^2)^2},\quad  x_3  = \frac{4\vartheta_8(v)^2 \vartheta_9(v)^2}{(\vartheta_8(v)^2 +\vartheta_9(v)^2)^2},
\end{gather}
where \(v=\per(x)\), \(\vartheta_j(v)\) is the pull-back of theta constant
\(\vartheta_{a,b}(\tau)\) with \((a,b)=\nu_j\)
under the embedding \(\imath:\mathbb{B}_3 \to \mathfrak{S}_4\), and
\(\nu_j\) are given as
\[
	\begin{array}{ccc}
		\nu_0 = (0000,0000), & \nu_1 = (1000,0100), & \nu_4 = (0010,0001), \\
		\nu_5 = (1010,0101), & \nu_8 = (0011,0000), & \nu_9 = (1011,0100).
	\end{array}
\]

We also have the relation
\begin{equation}\label{eq:thomae-quad-form}
	\left( \vartheta_0(v)^2+\vartheta_1(v)^2 \right)^2 = \frac{1}{256\pi^2\varGamma(3/4)^8} \left( \Transpose{v}Uv \right)^2. \end{equation}
This equality yields several relations among \(\vartheta_j(v)\).
By acting the matrix \(N \in \Sp(8,\mathbb{Z})\)
in \eqref{eq:mat-N} on the left-hand side of \eqref{eq:thomae-quad-form},
we have an analogue of Jacobi's formula:
\begin{equation}\label{eq:Thomae-formula}
	\vartheta_{0000,0000}(N\cdot \tau(v))^2 +\vartheta_{1100,0000}(N\cdot \tau(v))^2  = \frac{\pi}{\varGamma(3/4)^4}F_D\left( \frac{1}{4},\frac{1}{4},\frac{1}{4},\frac{1}{4},1;x_1,x_2,x_3\right)^2.
\end{equation}

We give functions \(a(v)\), \(b_1(v)\), \(b_2(v)\), \(b_3(v)\) on
\(\mathbb{B}_3\) as homogeneous quadratic polynomials of
\(\vartheta_j(N\cdot \tau(v))\).
We use transformation formulas of theta constants for \(\Sp(8,\Q)\)
in \cite[Lemma 5]{MMN07}
to show \eqref{eq:Intro-means}.
By combining the results obtained, we show 
\eqref{eq:Matsumoto AGM}
similarly to 
the previously mentioned proof of \eqref{eq:AGM=F} by 
\eqref{eq:2tau formula} and \eqref{eq:Jacobi-formula}.

\section{A Family of Curves}\label{sec:curve}
\subsection{The configuration space}
\begin{definition}
	\label{def:config-sp}
	The configuration space of five distinct points on \(\C\) is defined by
	the quotient space
	\[
		\mathfrak{X}_5= (\C^\times \ltimes \C) \backslash (\C^5-\Diag),
	\]
	where
	\[
		\Diag=\{\tilde x=(\tilde x_0,\tilde x_1,\dots,\tilde x_4)\in \C^5
		\mid \tilde x_j= \tilde x_k\ \textrm{for some }
		(0\le j<k\le 4)\},
	\]
	and \(\C^\times \ltimes \C\) is the semi-direct product of
	the multiplicative group \(\C^\times=\C-\{0\}\) and the additive group \(\C\),
	and acts on \(\C^5-\Diag\) as the affine transformation group by
	\[
		(\C^\times \ltimes \C)\times (\C^5-\Diag)\ni ((q_1,q_2),\tilde x)
		\mapsto( q_1\tilde x_0+q_2,\dots,q_1\tilde x_4+q_2) \in \C^5-\Diag.
	\]
	A  complete set of representatives for \(\mathfrak{X}_5\) is given by
	\begin{align*}
		 & X_5=\{(0,x,1)\in \C^5\mid x\in X\},                     \\
		 & X=\{x=(x_1,x_2,x_3)\in \C^3\mid x_j\ne 0,1\ (j=1,2,3),\
		x_j\ne x_k\ (1\le j<k\le 3)\}.
	\end{align*}
\end{definition}
For an element \(\tilde x=(\tilde x_0,\tilde x_1,\dots,\tilde x_4)\in
\C^5-\Diag\) and  four distinct elements \(n_1,n_2,n_3,n_4\) in
\(\{0,1,\dots,4\}\) satisfying \(n_1<n_2\) and \(n_1<n_3<n_4\),
we define \(15\) polynomials \(\mathcal{P}_{n_1,n_2;n_3,n_4}(\tilde x)\)
by
\begin{equation}\label{eq:15polys}
	\mathcal{P}_{n_1,n_2;n_3,n_4}(\tilde x)=(\tilde x_{n_2}-\tilde x_{n_1})\cdot
	(\tilde x_{n_4}-\tilde x_{n_3}).
\end{equation}
\begin{proposition}
	The map
	\[
		\C^5-\Diag\ni  \tilde x\mapsto
		(\dots,\mathcal{P}_{n_1,n_2;n_3,n_4}(\tilde x),\dots) \in \mathbb{C}^{15}
	\]
	induces an embedding from the quotient space \(\mathfrak{X}_5\)
	into the projective space  \(\Ps^{15-1}\).
\end{proposition}
\begin{proof}
	Since each \(\mathcal{P}_{n_1,n_2;n_3,n_4}(\tilde x)\) is multiplied by \(q_1^2\)
	under the action of \((q_1,q_2)\in \C^\times \times \C\),
	the map is regarded as a well-defined map from  \(\mathfrak{X}_5\) to \(\Ps^{15-1}\).
	Refer to \cite[\S 1.2]{DO} for a proof that it is an embedding.
\end{proof}
\subsection{Fourfold covering \(C(x)\) of \texorpdfstring{\(\mathbb{P}^1\)}{P1} branching at six points}

For an element \(\tilde x=(\tilde x_0,\tilde x_1,\dots,\tilde x_4)\in \C^5-\Diag\),
we define an affine algebraic curve
\begin{equation}
	\label{eq:affinecurve}
	C(\tilde x)=
	\{(z,w)\in \C^2\mid  w^4=\prod_{j=0}^4(z-\tilde x_j)\}.
\end{equation}
If \(\tilde x^\prime=(\tilde x^\prime_0,\tilde x_1^\prime\dots,\tilde x^\prime_4)\in \C^5-\Diag\)
is equivalent to
\(\tilde x\) under the action of \(\C^\times\ltimes \C\)
then the affine algebraic curve for \(\tilde x^\prime\) is isomorphic to
that for \(\tilde x\). In particular, these curves are isomorphic to
\begin{equation}
	\label{eq:N-affinecurve}
	C(x)=\{(z,w)\in \C^2\mid  w^4=z(z-x_1)(z-x_2)(z-x_3)(z-1)\}
\end{equation}
if \(\tilde x\) is equivalent to \((0,x_1,x_2,x_3,1)\) (\((x_1,x_2,x_3)\in X\)).

\begin{definition}
	We define a compact Riemann surface by
	the non-singular projective model of the affine algebraic curve \(C(x)\)	in \eqref{eq:N-affinecurve}.
	This compact Riemann surface is denoted by the same symbol \(C(x)\)
	as the affine algebraic curve.
	We set a family
	\[
		\mathcal{C}=\bigcup_{x\in X} C(x)
	\]
	of compact Riemann surfaces \(C(x)\) over \(X\).
\end{definition}
We define the projection \(\pr\) from the non-singular projective model \(C(x)\)
to the complex projective line \(\Ps^1\) by
\[
	\pr \colon C(x)\ni(z,w) \mapsto z\in \mathbb{C}.
\]
It is easy to see the following.

\begin{proposition}
	The pair \((C(x), \pr)\) defines a fourfold covering of
	the complex projective line \(\Ps^1\) branching at six points
	\(0,x_1,x_2, x_3,1, \infty\).
\end{proposition}
We set \( P_{0}=\pr^{-1}(0)\), \( P_{x_j} =\pr^{-1}(x_j)\) \((j=1,2,3)\),
\( P_{1}=\pr^{-1}(1)\), and \(P_\infty = \pr^{-1}(\infty)\) in \( C(x)\).
Note that each ramification index of these ramification points is four,
and that the local structure of the covering \(\pr\) around \(P_{\infty}\)
is different from that around any other ramification point.
\begin{proposition}
	The genus of \(C(x)\) is six.
\end{proposition}
\begin{proof}
	By the ramification indices of the fourfold covering \(\pr:C(x)\to \Ps^1\),
	the Euler characteristic of \(C(x)\) is given as \(4\times 2-6\times (4-1)=-10\),
	which yields 	the claim.
\end{proof}
\begin{definition}
	We define an automorphism \(\rho\) of the non-singular projective model \(C(x)\)
	by \[
		\rho\colon  C(x )\ni (z,w)\mapsto (z,i w)\in C(x).
	\]
	The automorphism \(\rho\) is of order \(4\), and satisfies \(\pr\circ \rho=\pr\).
	The automorphism \(\rho\) acts on the first homology group \(H_1(C(x),\Z)\)
	and the complex vector space \(H^0(C(x), \Omega^1)\)
	of holomorphic 1-forms on the compact Riemann surface \(C(x)\).
	The \((-1)\)-eigenspaces of the involution \(\rho^2\) on these spaces are
	denoted by
	\({H_1^-(C(x), \mathbb{Z})}\) and \({H^0_-(C(x), \Omega^1)}\), respectively.
\end{definition}

\subsection{Holomorphic \(1\)-Forms and \(1\)-Cycles}
We can easily show the following.
\begin{proposition}
	The space \(H^0(C(x), \Omega^1)\) is spanned by the holomorphic
	differential \(1\)-forms
	\[
		\phi_1=\frac{dz}{w},\quad \phi_2=\frac{dz}{w^3},\quad \phi_3=\frac{z\,dz}{w^3},\quad \phi_4=\frac{z^2\,dz}{w^3}, \quad \frac{dz}{w^2},\quad \frac{z\,dz}{w^2},
	\]
	on the affine algebraic curve \(C(x)\).
	The holomorphic \(1\)-forms \(\phi_1,\dots,\phi_4\) on \(C(x)\)
	form a basis of the \((-1)\)-eigenspace \({H^0_-(C(x), \Omega^1)}\)
	of \(\rho^2\).
\end{proposition}

We take an element \(x=(x_1,x_2,x_3)\in X_{\mathbb{R}}^{123}\).
We define paths \(L_j\) \((j=1,\dots,6)\) in \(C(x)\)
connecting two ramification points so that
\(\pr(L_j^\circ)\) and \(\arg(w)\) are given in Table \ref{tab:paths},
where \(L_j^\circ\) is \(L_j\) with its endpoints removed.
\begin{table}[hbt]
	\[
		\begin{array}{|c||c|c|c|c|c|c|}
			\hline
			\textrm{paths} & L_1            & L_2       & L_3            & L_4            & L_5            & L_6 \\
			\hline
			\pr(L_j^\circ) &
			(-\infty, 0)   & (0,x_1)        & (x_1,x_2) & (x_2,x_3)      & (x_3,1)        & (1,\infty)
			\\
			\hline
			\arg(w)        & \frac{5}{4}\pi & \pi       & \frac{3}{4}\pi & \frac{1}{2}\pi & \frac{1}{4}\pi
			& 0                                                                                   \\
			\hline
		\end{array}
	\]
	\caption{Paths \(L_j\)}
	\label{tab:paths}
\end{table}
Here, note that any value \(w=\sqrt[4]{z(z-x_1)(z-x_2)(z-x_3)(z-1)}\)
on \(L_j\) is given by the analytic
continuation of that on \(L_6\) via the upper half-plane
\(\{z\in \C\mid \Im(z)>0\}\) in the \(z\)-coordinate space.

We have \(1\)-cycles
\[
	(1-\rho)\cdot L_j,
\]
which is the path joining \(L_j\) to the reverse path of \(\rho\cdot L_j\).
Here, note that the boundaries of \(L_j\) and \(\rho \cdot L_j\) are canceled in this construction.
We can easily show the following.
\begin{lemma}
	\label{lem:12cycles}
	The set  \(\Sigma =\{\rho^k (1-\rho)\cdot L_j \mid j=2,3,4,5,\ k=0,1,2\}\)
	spans 	\(H_1(C(x),\mathbb{Z})\).
	Moreover, the following identity holds:
	\[
		(1+\rho+\rho^2+\rho^3)\cdot\gamma = 0 \text{ in }H_1(C(x),\mathbb{Z})\textrm{ for any } \gamma \in H_1(C(x),\mathbb{Z}).
	\]
\end{lemma}

We define \(1\)-cycles
\[
	c_j = (1-\rho^2)\cdot  L_j \quad (j=1,\ldots,6),
\] and set
\(\Sigma^- = \{c_2,\ldots,c_5,\rho\cdot  c_2,\ldots,\rho\cdot   c_5\}\).
Note that
\[
	\rho^2\cdot c_j=\rho^2\cdot(1-\rho^2)\cdot L_j=(\rho^2-1)\cdot  L_j=-c_j.
\]
\begin{proposition}
	\label{prop:(-1)eigencycles}
	The set \(\Sigma^-\) forms a basis for \({H_1^-(C(x),\mathbb{Z})}\). Moreover, the intersection matrix \(Q\) with respect to the basis \(\Sigma^-\) is
	\[
		Q =
		\begin{pmatrix}
			Q_1  & Q_2 \\
			-Q_2 & Q_1
		\end{pmatrix},\ Q_1=
		\begin{pmatrix}
			0  & 1  & 0  & 0 \\
			-1 & 0  & 1  & 0 \\
			0  & -1 & 0  & 1 \\
			0  & 0  & -1 & 0
		\end{pmatrix},\ Q_2=
		\begin{pmatrix}
			2  & -1 & 0  & 0  \\
			-1 & 2  & -1 & 0  \\
			0  & -1 & 2  & -1 \\
			0  & 0  & -1 & 2
		\end{pmatrix}.
	\]
\end{proposition}
\begin{proof}
	By results in \cite[\S2.1]{MT04},
	we can show that the intersection matrix of \(\Sigma^-\) is \(Q\),
	which satisfies \(\det(Q)=2^4\).
	Since \(C(x)\) is a double cover of the hyperelliptic curve of genus \(2\) given by \(w^2 = z(z - x_1)(z - x_2)(z - x_3)(z - 1)\)  with \(6\) ramification points
	\(P_{0},P_{x_1},P_{x_2},P_{x_3},P_{1},P_\infty\),
	the dimension of \(\Prym(C(x), \rho^2)\) is \(4\)
	and the determinant of the intersection matrix of a basis of
	\(H^-_1(C(x),\Z)\) is \(2^4\) by results in \cite[Chapter V]{Fay}.
	Thus, \(\Sigma^-\) becomes a basis of \(H^-_1(C(x),\Z)\).
\end{proof}
\begin{proposition}
	\label{prop:lin-exp-c1c6}
	The cycles \(c_1\) and \(c_6\) can be expressed by linear combinations
	of the basis \(\Sigma^-\) over \(\Z\) as
	\[
		c_1 = -c_4-c_5+\rho\cdot c_3 +\rho\cdot  c_4,\quad
		c_6 = -c_2-c_3-\rho\cdot  c_3 - \rho\cdot  c_4.
	\]
\end{proposition}
\begin{proof}
	We express \(c_1\) as a linear combination
	\[
		c_1 = \sum_{j=2}^5 (r_j\,c_j + r_{j+4}\,\rho\cdot  c_j)
		=(r_2,\dots,r_9)\Transpose{(c_2,\dots,c_5,\rho\cdot  c_2,\dots,\rho\cdot  c_5)}
	\]
	using unknown integers \(r_2,\dots,r_9\).
	By computing the intersection numbers \(c_1\cdot (\rho^j\cdot  c_k)\)
	for \(k=2,\dots,5\), \(j=0,1\), we have  a  system of linear equations
	\((r_2,\dots,r_9)\,Q = -(c_1\cdot c_2,\dots,c_1\cdot(\rho\cdot c_5))\),
	whose solution gives the expression of \(c_1\).
	We can similarly obtain the expression of \(c_6\).
\end{proof}

We see that the set \(\{c_6,c_1,c_3,c_4,\rho\cdot c_6,\rho\cdot c_1,\rho\cdot  c_3,\rho\cdot c_4\}\) also spans \({H_1^-(C(x),\mathbb{Z})}\)
by computing the transition matrix from \(\Sigma^-\) to this set.

\begin{definition}
	We set    \(\Sigma_U^- = \{A_1,\dots,A_4,B_1,\dots,B_4\}\), where
	\begin{align}
		& A_1 = (1+\rho)c_1, &  & A_2 = \rho c_6,    &  & A_3 = -(1+\rho)c_3-\rho c_4, &  & A_4 = c_4,       \\
		& B_1 = c_6,         &  & B_2 = (1-\rho)c_1, &  & B_3 = -(1-\rho)c_3 - c_4,    &  & B_4 = -\rho c_4.
	\end{align}
	We define a sublattice \(\Lambda=\Lambda(x)\) of \({H_1^-(C(x),\mathbb{Z})}\)
	by the \(\Z\)-span of \(\Sigma_U^-\).
\end{definition}
\begin{proposition}
	\label{prop;symplectic-basis}
	The intersection matrix with respect to the basis \(\Sigma_U^-\) is
	\(2J_8\),
	where we set
	\[
		J_{2n} =
		\begin{pmatrix}
			O_n & -I_n \\
			I_n & O_n
		\end{pmatrix},\quad I_n=\diag(1,\dots,1),\quad
		O_n=\diag(0,\dots,0),
	\]
	and \(\diag(j_1,\dots,j_n)\) denotes the diagonal matrix of size \(n\) with
	diagonal entries \(j_1,\dots,j_n\).
	Moreover, the representation matrix of the action of \(\rho\) on \(\Lambda(x)\) is given by
	\[
		\begin{pmatrix} \rho(A) \\ \rho(B)
		\end{pmatrix}=
		\begin{pmatrix}
			O_4 & -U  \\
			U   & O_4
		\end{pmatrix}\begin{pmatrix} A \\ B
		\end{pmatrix},
	\]
	where
	\[
		A=\begin{pmatrix} A_1 \\\vdots  \\A_4
		\end{pmatrix},\quad
		B=\begin{pmatrix} B_1 \\ \vdots \\B_4
		\end{pmatrix},\quad
		U =
		\begin{pmatrix}
			0 & 1 & 0 & 0 \\
			1 & 0 & 0 & 0 \\
			0 & 0 & 1 & 0 \\
			0 & 0 & 0 & 1
		\end{pmatrix}.
	\]
	In particular, the sublattice \(\Lambda(x)\) in \({H_1^-(C(x),\mathbb{Z})}\)
	is of index \(4\).
\end{proposition}
\begin{proof}
	The elements in \(\Sigma^-_U\) are expressed linear combinations of
	those in \(\Sigma^-\) 	as
	\[
		\Transpose{(A_1,\dots,A_4,B_1,\dots,B_4)}=
		T_U\Transpose{(c_2,\dots,c_5,\rho c_2,\dots,\rho c_5)},
	\]
	where
	\begin{equation}
		\label{eq:sigma minus to Lx}
		T_U =
		\begin{pmatrix}
			T_1  & T_2   \\
			UT_2 & -UT_1
		\end{pmatrix},\ T_1=
		\begin{pmatrix}
			0 & -1 & -2 & -1 \\
			0 & 1  & 1  & 0  \\
			0 & -1 & 0  & 0  \\
			0 & 0  & 1  & 0
		\end{pmatrix},\ T_2=
		\begin{pmatrix}
			0  & 1  & 0  & -1 \\
			-1 & -1 & 0  & 0  \\
			0  & -1 & -1 & 0  \\
			0  & 0  & 0  & 0
		\end{pmatrix}.
	\end{equation}
	This matrix \(T_U\) yields the intersection matrix \(2J_8\) and the representation
	matrix of \(\rho\) for the basis \(\Transpose{(A_1,\dots,A_4,B_1,\dots,B_4)}\).
	Since \(\det T_U =-4\),
	\(\Lambda(x)\) is a sublattice of index \(4\)
	in \({H_1^-(C(x),\mathbb{Z})}\).
\end{proof}
We rename the element \(x\in X_{\mathbb{R}}^{123}\) taken at the beginning of this subsection
to \(\dot x\), and fix it as a base point.
\begin{definition}
	For any element \(x\in X\), we choose a path \(\ell\) from \(\dot x\) to \(x\).
	We define cycles \(L_j,c_j\) \((j=1,\dots,6)\), \(A_k,B_k\) \((k=1,\dots,4)\)
	on \(C(x)\) by the continuations of the cycles on \(C(\dot x)\) along
	the path \(\ell\).
\end{definition}
\begin{remark}
	Though the continuations depend on the choice of paths,
	they satisfy the properties in Lemma \ref{lem:12cycles} and
	Propositions \ref{prop:(-1)eigencycles}, \ref{prop:lin-exp-c1c6} and
	\ref{prop;symplectic-basis} by the local triviality of the continuation.
\end{remark}

\subsection{Period matrices}
\begin{definition}
	We define the period matrix \(\Pi\) with respect to
	the symplectic basis \(\Sigma_U^-\) and the basis \(\{\phi_1,\phi_2,\phi_3,\phi_4\}\) by
	\[
		\Pi = \begin{pmatrix} \tau_A\\ \tau_B\end{pmatrix},\quad
		\tau_A=\Big(\int_{A_j} \phi_k\Big)_{j,k},\quad
		\tau_B=\Big(\int_{B_j} \phi_k\Big)_{j,k}.
	\]
\end{definition}

We can show the following proposition similarly to  classical
Riemann's bilinear relations and inequalities.
\begin{proposition}
	The period matrix \(\Pi\) satisfies  bilinear relations
	and inequalities: 	\begin{equation}
		\label{eq:Riemann-rel}
		\Transpose{\Pi}J_8\Pi = O_4,\quad i\Transpose{\Pi}J_8
		\overline{\Pi} > 0,
	\end{equation}
	where \(M>0\) means that a Hermitian matrix \(M\) is positive definite.
	In particular, \(\tau_A\) and \(\tau_B\) are invertible,
	and the normalized period matrix \(\tau=\tau_A\tau_B^{-1}\) lies
	in the Siegel upper half-space \(\mathfrak{S}_4\),
	where \(\mathfrak{S}_n\) consists of symmetric matrices in
	\(\GL(n,\mathbb{C})\) with positive-definite imaginary part.
\end{proposition}

We set
\[
	\varphi = (\varphi_1,\varphi_2,\varphi_3,\varphi_4)=(\phi_1,\phi_2,\phi_3,\phi_4)\tau_B^{-1} ,
\]
which is a normalized basis of \(H^0_-(C(x),\Omega^1)\) satisfying
\[
	\Big(\int_{A_j} \varphi_k\Big)_{j,k}=\tau,\quad
	\Big(\int_{B_j} \varphi_k\Big)_{j,k}=I_4.
\]
\begin{proposition}
	The normalized period matrix \(\tau\) satisfies
	\[
		(U\tau)^2 = -I_4,\quad \det(\tau) = 1.
	\]
\end{proposition}
\begin{proof}
	Since \(U^2 = I_4\), \(\tau = \tau_A \tau_B^{-1}\) and \(\tau_B = -U\tau_A \diag(-i,i,i,i)\) by \(B=-U\rho(A)\) and
	\((\rho^\ast (\phi_1),\dots,\rho^\ast (\phi_4))=(\phi_1,\dots,\phi_4)\diag(-i,i,i,i)\),
	we see that
	\begin{align}
		\left( U\tau \right)^2
		& = U\tau_A \tau_B^{-1} U \tau_A \tau_B^{-1}= U\tau_A (-\diag(-i,i,i,i)^{-1} \tau_A^{-1} U)U \tau_A \tau_B^{-1} \\
		& = U\tau_A \diag(-i,i,i,i) \tau_B^{-1} = -\tau_B \tau_B^{-1} = -I_4.
	\end{align}
	Since \(\det(\tau_B) =\det( -U\tau_A \diag(-i,i,i,i))=\det(\tau_A)\),
	we have \(\det(\tau)=1\).
\end{proof}

\subsection{An embedding of \(\mathbb{B}_3\)  into  \(\mathfrak{S}_4\)}

We take the first column vector \(v=\Transpose{(v_1,\dots,v_4)}\)
of \(\tau_B \). Each entry of \(v\) is a linear combination of integrals
\[
	\int_{c_j} \phi_1=\int_{c_j} \frac{dz}{\sqrt[4]{z(z-1)(z-x_1)(z-x_2)(z-x_3)}}.
\]
To represent some entries of \(v\) by power series,
we introduce the Lauricella hypergeometric series
\begin{equation}\label{eq:HGS-FD}
	F_D\left( \alpha,\beta,\gamma;z\right) =
	\sum_{n_1,\ldots,n_m \geq 0} \frac{(\alpha,\sum_{j=1}^m n_j)\prod_{j=1}^m (\beta_j,n_j)}{(\gamma,\sum_{j=1}^m n_j)\prod_{j=1}^m (1,n_j)} \prod_{j=1}^m z_j^{n_j}
\end{equation}
of type \(D\) in \(m\) variables \(z = (z_1,\ldots,z_m)\)
with complex parameters \(\alpha,\beta = (\beta_1,\ldots,\beta_m),
\gamma\) \((\ne 0,-1,-2,\dots)\), where
\((\alpha,n)  = \alpha(\alpha+1)\cdots(\alpha+n-1)=
\varGamma(\alpha+n)/\varGamma(\alpha)\).
It converges absolutely on the set
\(\mathbb{D}^m=\{z\in \mathbb{C}^m\mid |z_j|<1\ (j=1,\dots,m)\}\),
and admits an Euler type integral representation
\begin{equation}\label{eq:Euer-int-rep}
	F_D\left( \alpha,\beta,\gamma;z\right)
	=\frac{\varGamma(\gamma)}{\varGamma(\alpha)\varGamma(\gamma-\alpha)}
	\int_1^\infty t^{\beta_1+\cdots+\beta_m-\gamma}(t-1)^{\gamma-\alpha}
	\prod_{j=1}^m (t-z_j)^{-\beta_j} \frac{dt}{t-1}
\end{equation}
under the condition \(0<\Re(\alpha)<\Re(\gamma)\).
\begin{lemma}\label{lem:v1v2-FD}
	If \(x\in \mathbb{D}^3\), then
	\begin{equation}
		\label{eq:v1-HGS}
		v_1  		= 2\sqrt{2}\pi F_D\left( \frac{1}{4},\frac{1}{4},\frac{1}{4},\frac{1}{4},1;x_1,x_2,x_3 \right).
	\end{equation}
	If \((1-x_1,1-x_2,1-x_3)\in \mathbb{D}^3\), then
	\begin{equation}
		\label{eq:v2-HGS}
		v_2  		= -4\pi F_D\left( \frac{1}{4},\frac{1}{4},\frac{1}{4},\frac{1}{4},1;1-x_1,1-x_2,1-x_3 \right).
	\end{equation}
\end{lemma}
\begin{proposition}\label{prop:expression of period matrix by the period}
	Let \(v\) be the first column vector of \(\tau_B \).
	Then, \(\Transpose{v}Uv \neq 0\) and the normalized period matrix
	\(\tau = \left( \int_{A_j} \varphi_k \right)_{j,k}\) is given by
	\[
		\tau = iU \left( I_4 -\frac{2}{\Transpose{v}Uv}v \Transpose{v} U \right).
	\]
	Moreover, the vector \(v\) satisfies \(v^\ast Uv < 0\).
\end{proposition}
\begin{proof}
	Set \(\tau_A=(u_1,\dots,u_4)\) and \(\tau_B=(u_1',u_2',u_3',u_4')\).
	Then we have \(u_1'=v\) and
	\[
		(u_1',u_2',u_3',u_4')
		=(iUu_1,-iUu_2,-iUu_3,-iUu_4).
	\]
	From this relation, the column vectors satisfy
	\[
		\Transpose{v}(Uu_j') = -i\Transpose{v}u_j,
		\quad
		(\Transpose{v}U)u'_j = i\Transpose{u_1}u'_j\quad (j=2,3,4).
	\]
	The equality in \eqref{eq:Riemann-rel}
	yields
	\[
		\Transpose{u_j}u'_k - \Transpose{u'_j}u_k = 0
		\quad(1\le j,k\le4),
	\]
	hence we have
	\[
		2\Transpose{v}Uu'_j = i \left( \Transpose{u_1}u'_j - \Transpose{v}u_j \right)=0,
	\]
	so \(\Transpose{v}Uu'_j=0\) hold for \(j=2,3,4\).
	It follows that \(\Transpose{v}Uv \neq 0\) since \(\tau_B\) is invertible.
	Note also that \(u'_2\), \(u'_3\), \(u'_4\) are \(i\)-eigenvectors of \(i\left( I_4 - \tfrac{2}{\Transpose{v}Uv}v\Transpose{v}U \right)\) and \(v\) is a  \((-i)\)-eigenvector of this matrix.

	On the other hand, the vector \(v\) is a \((-i)\)-eigenvector of
	\[
		U\tau = U\tau_A\tau_B^{-1}
		= i\tau_B\diag(-1,1,1,1)\tau_B^{-1},
	\]
	and \(u'_2\), \(u'_3\), \(u'_4\) are \(i\)-eigenvectors of \(U\tau\).
	Hence, we have
	\[
		U\tau=i\left( I_4 - \frac{2}{\Transpose{v}Uv}v\Transpose{v}U \right)
	\]
	by the coincidence of the eigenspaces of these matrices.
	This equality yields the expression of \(\tau\).

	Since \(\Im\tau\) is positive definite, we have
	\[
		\begin{aligned}
			0 < v^\ast (\Im\tau) v
			& = v^\ast\left( U - 2\Re\left( (\Transpose{v}Uv)^{-1}Uv\Transpose{v}U \right) \right)v \\
			& = v^\ast Uv -\left( v^\ast Uv + v^\ast Uv \right)
			= -v^\ast Uv,
		\end{aligned}
	\]
	hence \(v\) satisfies \(v^\ast Uv<0\).
\end{proof}
We define a domain \(\mathcal{B}\) in \(\C^4\) and
the \(3\)-dimensional complex ball \(\mathbb{B}_3\) by
\begin{equation}\label{eq:pre-ball}
	\mathcal{B}=\{v\in \C^4\mid v^\ast Uv<0\}, \quad \mathbb{B}_3=\mathcal{B}/
	\C^\times.
\end{equation}
By Proposition \ref{prop:expression of period matrix by the period},
the first column vector \(v\) of \(\tau_B\) is in \(\mathcal{B}\)
and represents an element of \(\mathbb{B}_3\).
Hereafter, we use the same symbols for an element in \(\mathcal{B}\)
and for the equivalence class containing it in \(\mathbb{B}_3\)
when there is no risk of confusion.

\begin{lemma}\label{lem:tvUv-nonzero}
	Every element \(v\in \mathcal{B}\) satisfies \(\Transpose{v}Uv\ne 0\).
\end{lemma}
\begin{proof}
	Take \(g\in \GL(4,\R)\) satisfying \(U=\Transpose{g}\diag(-1,1,1,1)g\),
	and set \(w=\Transpose{(w_1,\dots,w_4)}=gv\) for \(v \in \mathcal{B}\).
	Then the condition \(v^\ast Uv<0\) is equivalent to
	\(|w_1|^2>|w_2|^2+|w_3|^2+|w_4|^2\).
	Thus we have
	\[
		|\Transpose{v}Uv|=|-w_1^2+w_2^2+w_3^2+w_4^2|
		\ge |w_1|^2-\left( |w_2|^2+|w_3|^2+|w_4|^2 \right)>0,
	\]
	which yields \(\Transpose{v}Uv\ne 0\).
\end{proof}

The following proposition shows that any element \(v\in \mathbb{B}_3\) gives an element of \(\mathfrak{S}_4\) by \eqref{eq:def of tau}, refer to  \cite{Na26} for its proof.

\begin{proposition}
	\label{prop:equiv-B3S4}
	We define a \(4\times 4\) matrix
	\begin{equation}
		\label{eq:def of tau}
		\tau(v) = iU \left( I_4 -\frac{2}{\Transpose{v}Uv}v \Transpose{v} U \right),
	\end{equation}
	for a vector \(v \in \mathbb{C}^4\) satisfying \(\Transpose{v} U v\ne 0\).
	Then, it is invariant under the right action of \(\C^\times\) on \(v\),
	and the following statements are equivalent:
	\begin{enumerate}
		\item the vector \(v\) is in  \(\mathcal{B}\),
		\item the matrix \(\tau(v)\) is  in the Siegel upper half-space \(\mathfrak{S}_4\).
	\end{enumerate}
\end{proposition}

Hence, we can define an embedding
\begin{equation}\label{eq:imath}
	\imath\colon \mathbb{B}_3 \ni v \mapsto \tau(v) \in \mathfrak{S}_4.
\end{equation}
By straightforward calculations of matrices,
we have the following proposition.

\begin{proposition}
	\label{prop:modularembed}
	The embedding \(\imath\) induces a homomorphism
	\begin{equation}
		\label{map:def of emb between auto grp}
		\jmath\colon
		\U(U,\mathbb{C}) \ni g \mapsto
		\begin{pmatrix}
			U \Re(g) U & U \Im(g) \\
			-\Im(g)U   & \Re(g)
		\end{pmatrix} \in \Sp(8,\mathbb{R}),
	\end{equation}
	where we define the unitary group \(\U(U,\mathcal{R})\) over a subring \(\mathcal{R} \subset \C\)
	and the symplectic group \(\Sp(2n,\mathcal{R}^\prime)\) over a subring \(\mathcal{R}^\prime \subset \R\)
	by
	\begin{align*}
		\U(U,\mathcal{R})          & = \{g \in \GL(4,\mathcal{R}) \mid g U g^\ast =U\},               \\
		\Sp(2n,\mathcal{R}^\prime) & = \{M \in \GL(2n,\mathcal{R}^\prime) \mid MJ_{2n}\Transpose{M} =
		J_{2n}\}.
	\end{align*}
	The maps \(\imath\) and \(\jmath\) satisfy
	\begin{equation}
		\label{eq:modular embedding}
		\imath(g\cdot v)=\jmath(g)\cdot \imath(v)
	\end{equation}
	for any \(g\in \U(U,\C)\) and any \(v\in \mathbb{B}_3\), where \(\Sp(2n,  \mathcal{R}^\prime)\) acts on
	the Siegel upper half-space \(\mathfrak{S}_n\) by
	\[
		\Sp(2n,  \mathcal{R}^\prime)\times \mathfrak{S}_n\ni
		\left(M,\tau_n\right) \mapsto
		(M_{11}\tau_n + M_{12})(M_{21}\tau_n + M_{22})^{-1}\in \mathfrak{S}_n,
	\]
	where \(M=\begin{pmatrix} M_{11} & M_{12} \\ M_{21}& M_{22}
	\end{pmatrix}\in \Sp(2n,\mathcal{R}^\prime)\).
\end{proposition}

\subsection{Period maps}
\label{sec:Period maps}
By Proposition \ref{prop:equiv-B3S4}, we have a single-valued holomorphic map
from a neighborhood \(V\) of \(\dot x\in X\) to \(\mathbb{B}_3\):
\[
	\per\colon V\ni x\mapsto v=
	\Transpose{\left( \int_{B_1}\phi_1,\dots, \int_{B_4}\phi_1 \right)}
	\in\mathbb{B}_3.
\]
We can extend it to the map \(\widetilde{\per}\) from
the universal covering \(\tilde X\) of \(X\) to \(\mathbb{B}_3\)
by analytic continuation.
This extension induces
a homomorphism \(\mu \) from the fundamental group
\(\pi_1(X,\dot x)\) to \(\GL_4(\C)\) :
\[
	\mu \colon \pi_1(X,\dot x) \ni \gamma\mapsto \mu (\gamma)
	\in \GL(4,\C),
\]
where \(\gamma\) is a loop  in \(X\) with base point \(\dot x\), and
\(\mu (\gamma)\) is the circuit matrix of \(\gamma\) with respect
to \(v\), that is, the analytic continuation of \(v\) along \(\gamma\)
is expressed by \(\mu (\gamma )v.\)
Here, for two loops \(\gamma\) and \(\gamma^\prime\) with base point \(\dot x\),
\(\gamma\cdot \gamma^\prime\) denotes the loop joining the start point of
\(\gamma^\prime\) to the end point of \(\gamma\), and their circuit matrices
satisfy
\[
	\mu (\gamma\cdot \gamma^\prime)=\mu (\gamma)
	\mu (\gamma^\prime).
\]
The image of \(\pi_1(X,\dot x)\) under the map \(\mu \) is called
the monodromy group of \(\per\), and it is denoted by \(\Gamma\).
By Proposition \ref{prop:modularembed}, we can see that
\(\Gamma\) is a subgroup of the unitary group \(\U(U,\C)\).
By taking the quotient of \(\mathbb{B}_3\) by the monodromy group \(\Gamma\),
we obtain a single-valued holomorphic map
\begin{equation}
	\label{eq:periodmap}
	\per\colon X \to \Gamma\backslash\mathbb{B}_3,
\end{equation}
which is called a period map for the family \(\mathcal{C}\).
By composing \(\per\) and the embedding \(\imath\colon \mathbb{B}_3 \hookrightarrow \mathfrak{S}_4\), we have
the map \(\imath\circ \per\colon X \to \mathfrak{S}_4\),
which is also called a period map.

It is shown in \cite[\S 4.5, 6.2]{Yos97} that
the image of the period map \(\widetilde{\per}(\tilde{X})\)
is isomorphic to an open dense subset
\((\mathbb{B}_3)^\circ\) in \(\mathbb{B}_3\),
and that the quotient space \(\Gamma\backslash (\mathbb{B}_3)^\circ\)
is isomorphic to \(X\).
The period map \eqref{eq:periodmap} can be extended to the map
from \(\Ps^3\) to the Satake-Baily-Borel compactification
\(\overline{\Gamma\backslash \mathbb{B}_3}\) of
\(\Gamma\backslash (\mathbb{B}_3)^\circ\),
which is given by \(\Gamma\backslash \mathbb{B}_3\)
plus five cusps corresponding to
\([1,0,0,0]\), \([0,1,0,0]\), \([0,0,1,0]\), \([0,0,0,1]\), \([1,1,1,1]\) in \(\Ps^3\).
Here, the space \(X\) is embedded into \(\Ps^3\) by
\[
	X\ni (x_1,x_2,x_3)\mapsto [x_1,x_2,x_3,1]\in \Ps^3.
\]
This extension is also denoted by \(\per\).
We have the diagram:
\[
	\begin{tikzcd}
		{\tilde{X}} && {\left( \mathbb{B}_3 \right)^\circ} && {\imath \left( \mathbb{B}_3 \right)^\circ} \\
		X && {\Gamma \backslash\left( \mathbb{B}_3 \right)^\circ} && {\jmath(\Gamma) \backslash \imath\left( \mathbb{B}_3 \right)^\circ.}
		\arrow["{\widetilde{\per}}", from=1-1, to=1-3]
		\arrow[from=1-1, to=2-1]
		\arrow[from=1-1, to=2-5]
		\arrow["\imath", from=1-3, to=1-5]
		\arrow[from=1-3, to=2-3]
		\arrow[from=1-5, to=2-5]
		\arrow["\per", from=2-1, to=2-3]
		\arrow["\imath"', from=2-3, to=2-5]
	\end{tikzcd}
\]

\subsection{Half-turn circuit matrices}
To study the inverse of \(\per\), we give half-turn circuit matrices
in this subsection.
For this purpose,  we prepare locally holomorphic
functions around \((0,\dot x, 1) \in \C^5-\Diag\).
We define functions \(u_k\) for \(k=1,2,3,4\) by
\[
	u_k(\tilde x) = \int_{\tilde x_{k-1}}^{\tilde x_k} \frac{dz}{w}
\]
for each element \(\tilde x=(\tilde x_0,\tilde x_1,\dots,\tilde x_4) \in \C^5-\Diag\),
where \((z,w)\) is a point of the curve \(C(\tilde x)\) in \eqref{eq:affinecurve}.
Then, the function \(u(\tilde x) = \Transpose{(u_1(\tilde x), \ldots, u_4(\tilde x))}\) is locally single valued and holomorphic around \((0,\dot x, 1)\in \C^5-\Diag\).

Let
\(\tilde x^{(j,k)} = (x_{(j,k)(0)},\ldots,x_{(j,k)(4)})\) be the element of \(\C^5 - \Diag\)
obtained by the action of the transposition \((j,k)\) of \(j\) and \(k\) on the indices of \(\tilde x \in \C^5 - \Diag\).
We analytically continue the function \(u(\tilde x)\) along
the path in \(\C^5-\Diag\) from \(\tilde x\) to \(\tilde x^{(j,k)}\)
given in \cite[\S4.5]{Yos97}. Then \(u(\tilde x^{(j,k)})\) is expressed as \(g^\prime_{j,k}u(\tilde x)\),
where  \(g^\prime_{j,k}\in \GL(4,\C)\).
By half-turn formulas in \cite[\S 4.5]{Yos97} taking care of
the difference in the branches of \(w\), we give \(g^\prime_{j,k}\) as in the
following proposition.

\begin{proposition}\label{prop:half turn formula} 						The matrices \(g^\prime_{j,j+1}\) \((j=0,1,2,3)\) are given by
	\begin{align}
		u_{j}(\tilde x^{(j,j\!+\!1)})   & = u_{j}(\tilde x) - i u_{j\!+\!1}(\tilde x)\ (j\ne0),
		& u_{j\!+\!1}(\tilde x^{(j,j\!+\!1)})                   & = i u_{j\!+\!1}(\tilde x),                  \\
		u_{j+2}(\tilde x^{(j,j\!+\!1)}) & = u_{j\!+\!1}(\tilde x) + u_{j+2}(\tilde x)\ (j\ne3),
		& u_k(\tilde x^{(j,j\!+\!1)})                           & = u_k(\tilde x) \ (|j\!+\!1\!-\! k| \ge 2).
	\end{align}
	Moreover, the matrices \(g^\prime_{j,k}\) are given by the conjugation
	\[
		(g^\prime_{k-1}\cdots g^\prime_{j+1})g^\prime_j(g^\prime_{k-1}\cdots g^\prime_{j+1})^{-1}\quad (j=0,1,2,\ k=j+2,\ldots,4)
	\]
	where \(g^\prime_k=g^\prime_{k,k+1}\).
\end{proposition}
We set
\(g_{j,k}=T_U^\prime g^\prime_{j,k} (T_U^\prime)^{-1}\)
by using \(T_U^\prime=U(T_2+iT_1)\)
for \(T_1\) and \(T_2\) given in \eqref{eq:sigma minus to Lx}.
Note that \(T_U^\prime\) is regarded as the transformation matrix
for the bases 
\(\{\int_{c_2}\phi_1,\ldots,\int_{c_5}\phi_1\}\)
and   
\(\{\int_{B_1}\phi_1,\ldots,\int_{B_4}\phi_1\}\)
of the space of local solutions to the hypergeometric system of rank \(4\)
satisfied by \(F_D(\frac{1}{4},\frac{1}{4},\frac{1}{4},\frac{1}{4},1;z)\).
We explicitly provide only the matrices \(g_{j,k}\)
which will be used in the following arguments:
\begin{align}
	\label{eq:half-turn-Mat}
	g_{0,1} & = \begin{pmatrix}
		1      & 0 & 0      & 0 \\
		-1 + i & 1 & -1 - i & 0 \\
		1 + i  & 0 & i      & 0 \\
		0      & 0 & 0      & 1 \\
	\end{pmatrix},                      &
	g_{1,2} & = \begin{pmatrix}
		1 & 0 & 0                 & 0                \\
		0 & 1 & 0                 & 0                \\
		0 & 0 & \dfrac{1 + i}{2}  & \dfrac{1 + i}{2} \\
		0 & 0 & -\dfrac{1 + i}{2} & \dfrac{1 + i}{2} \\
	\end{pmatrix}, \\ \nonumber
	g_{1,3} & = \begin{pmatrix}
		1 & 0 & 0                & 0                \\
		0 & 1 & 0                & 0                \\
		0 & 0 & \dfrac{1 + i}{2} & \dfrac{1 - i}{2} \\
		0 & 0 & \dfrac{1 - i}{2} & \dfrac{1 + i}{2} \\
	\end{pmatrix},  &
	g_{2,3} & = \begin{pmatrix}
		1 & 0 & 0 & 0 \\
		0 & 1 & 0 & 0 \\
		0 & 0 & 1 & 0 \\
		0 & 0 & 0 & i \\
	\end{pmatrix}.\end{align}
Note that the monodromy group \(\Gamma\) is generated by the matrices
\(g_{j,k}^2\) for \(0\le j<k\le 4\).

\section{Construction of the Inverse of the Period Map} \label{sec:construction of p}

To construct the inverse of the period map \(\per\colon \Ps^3 \to
\overline{\Gamma\backslash  \mathbb{B}_3} \),
we define the Abel-Jacobi-\(\Lambda\)  map and the theta function. By pulling back the theta function under the Abel-Jacobi-\(\Lambda\)  map, we construct  rational functions on \(C(x)\), and give relations between theta constants and
branch points.

\subsection{Action of Some Elements of the Symplectic Group on Theta Functions}
\begin{definition}\label{def:Rieman-theta}
	We define Riemann's theta function in variables
	\((\zeta,\tau_n) \in \mathbb{C}^n \times \mathfrak{S}_n\)
	with half characteristics \((a/2, b/2)\) by the series
	\begin{equation}\label{eq:Rieman-theta}
		\vt{a}{b}(\zeta,\tau_n) = \sum_{k\in \mathbb{Z}^n} \e \left( \frac{1}{2}\left( k\!+\!\frac{1}{2}a \right) \tau_n \Transpose{\left( k\!+\!\frac{1}{2}a \right)} + \left( k\!+\!\frac{1}{2}a \right) \Transpose{\left( \zeta\!+\!\frac{1}{2}b \right)}\right),
	\end{equation}
	where \(a, b \in \mathbb{Z}^n\) and
	\(\e(t) = \exp(2\pi i t)\).
	The series in \eqref{eq:Rieman-theta} converges absolutely and uniformly
	on any compact set in \(\mathbb{C}^n \times \mathfrak{S}_n\),
	and hence it is holomorphic on \(\mathbb{C}^n \times \mathfrak{S}_n\).
	It is also denoted by \(\vartheta_{a,b}(\zeta,\tau_n)\) or \(\vartheta_{m}(\zeta,\tau_n)\) for \(m=(a,b)\).
	The theta constant 	is defined by the value of \(\vt{a}{b}(\zeta,\tau_n)\) at \(\zeta = (0,\dots,0)\),
	and it is denoted by
	\(\vt{a}{b}(\tau_n)\), \(\vartheta_{a,b}(\tau_n)\) or  \(\vartheta_{m}(\tau_n)\).
\end{definition}
\begin{remark}
	In this paper, we consider only theta functions with half characteristics. Hence, we assume throughout that the characteristics are in
	\(\frac{1}{2}\mathbb{Z}^n\), and we omit denominators of characteristics
	in the notation of Riemann's theta function.
\end{remark}

\begin{definition}
	We define a holomorphic function \(\vartheta_{a,b}(v)\) on the complex ball
	\(\mathbb{B}_3\)
	by the pullback of the theta constant \(\vartheta_{a,b}(\tau_4)\) under
	the embedding \(\imath \colon \mathbb{B}_3 \to \mathfrak{S}_4\) given
	in \eqref{eq:imath}.
\end{definition}
We prepare the transformation formula of
the theta function with characteristics and that of the theta constant.
We define a map
by
\[
	\chi \colon \Sp(8,\mathbb{R})\times \mathfrak{S}_4 \ni
	(M,\tau)\mapsto \chi(M,\tau)= \det(M_{21}\tau+M_{22})\in \mathbb{C},
\]
where \(M=
\begin{pmatrix}
	M_{11} & M_{12} \\
	M_{21} & M_{22}
\end{pmatrix}\) and \(\tau \in \mathfrak{S}_4\).

\begin{lemma}[{\cite[Corollaries in p.85,176, Theorem 3 in p.182]{Igu72}}]
	\label{lem:Igusa-formula}
	Let \(\tau_n\) be a point of \(\mathfrak{S}_{n}\), and let \(M=\begin{pmatrix} M_{11} & M_{12}\\M_{21} & M_{22}\end{pmatrix}\) be an element of \(\Sp(2n,\mathbb{Z})\). For \(\zeta\in \C^n\) and \((a,b) \in \mathbb{Z}^{n}\times \mathbb{Z}^{n}\),
	set
	\begin{align}
		M \cdot \zeta=              & \zeta (M_{21}\tau_n + M_{22})^{-1},                                                                                                                             \\
		(a^\prime, b^\prime) =      & M \cdot (a,b)=(a,b)M^{-1}+
		((M_{21}\Transpose{M_{22}})_0,(M_{11}\Transpose{M_{12}})_0),
		\\
		\phi_{a,b}(M)             = & -\frac{1}{8} \left( a \Transpose{M_{22}} M_{12} \Transpose{a} - 2 a \Transpose{M_{12}} M_{21} \Transpose{b} + b \Transpose{M_{21}} M_{11} \Transpose{b} \right) \\
		& + \frac{1}{4} (a \Transpose{M_{22}} - b \Transpose{M_{21}}) \Transpose{(M_{11}\Transpose{M_{12}})_0},
	\end{align}
	where \((M^\prime)_0\) is the row vector consisting of the diagonal entries of a square matrix \(M^\prime\).
	\begin{enumerate}
		\item There exists a complex number \(\varsigma\) with absolute value \(1\), which is independent of \(\zeta\) and \(\tau_n\), such that
		      \begin{equation}\label{eq:tr-theta-function}
			      \vt{a^\prime}{b^\prime}(M \cdot\zeta, M \cdot\tau_n)
			      =\varsigma \e\big(\frac{1}{2}\zeta(M_{21}\tau_{n}+M_{22})^{-1}M_{21}\Transpose{\zeta}\big)
			      \chi(M,\tau_n)^{1/2} \vt{a}{b}(\zeta,\tau_n).
		      \end{equation}

		\item There exists  an eighth root \(\kappa(M)\)
		      of unity, whose square depends only on \(M\), 	such that
		      \begin{equation}\label{eq:tr-theta}
			      \vt{a^\prime}{b^\prime}(M \cdot \tau_n) = \kappa(M) \, \e\big(\phi_{a,b}(M)\big) \, \chi(M,\tau_n)^{1/2} 
			      \vt{a}{b}(\tau_n).
		      \end{equation}
	\end{enumerate}
\end{lemma}
\begin{proposition}\label{prop:eval Psi}
	For \(v\in \mathbb{B}_3\), \(g\in \U(U,\mathbb{C})\) and \(\jmath(g)=
	\begin{pmatrix}
		M_{11} & M_{12} \\
		M_{21} & M_{22}
	\end{pmatrix} \in \Sp(8,\mathbb{R})\), we have
	\begin{equation}\label{eq:auto-factor}
		\chi(\jmath(g),\imath(v)) =\det(M_{21}\imath(v)+M_{22})
		=\frac{\Transpose{(gv)}U(gv)}{\det(g)\Transpose{v}Uv}.
	\end{equation}
\end{proposition}
\begin{proof} We can show the assertion similarly to \cite[Lemma 1]{MMN07}.
\end{proof}
By applying \eqref{eq:tr-theta} to \(\tau\in \mathfrak{S}_4\) and
some elements  \(M\in \Sp(8,\Z)\),
we have the following lemma.
\begin{lemma} We give some transformation formulas for theta constants.
	\label{lem:monodromy action and U action}

	\begin{enumerate}
		\item  Let \(M=\jmath(g_{2,3})\) be the image of \(g_{2,3}\)
		      in \eqref{eq:half-turn-Mat}
		      under the map \(\jmath\) in \eqref{map:def of emb between auto grp} in Proposition \ref{prop:modularembed}.
		      Then \(M\) is in \(\Sp(8,\Z)\), and
		      \[
			      \vt{a^\prime}{b^\prime}(M\cdot \tau) = \frac{1+i}{\sqrt{2}} \e\left( -\frac{a_4 b_4}{4} \right) (-\tau_{44})^{1/2} \vt{a}{b}(\tau),
		      \]
		      where the argument of \(-\tau_{44}\) in \(-\mathbb{H} = \{z \in \mathbb{C} \mid \Im(z) < 0\}\) is supposed to be
		      \(-\pi < \arg(-\tau_{44}) < 0\).
		\item For \(M=\jmath(g_{2,3})^{-1}\),
		      \[
			      \vt{a^\prime}{b^\prime}(M\cdot \tau) = \frac{1-i}{\sqrt{2}} \e\left( -\frac{a_4 b_4}{4} \right) (\tau_{44})^{1/2} \vt{a}{b}(\tau),
		      \]
		      where the argument of \(\tau_{44}\) in \(\mathbb{H}\) is
		      supposed to be \(0 < \arg(\tau_{44}) < \pi\).
		\item For \(M=
		      \begin{pmatrix}
			      O_4 & -I_4 \\
			      I_4 & O_4
		      \end{pmatrix}\),
		      \begin{equation}
			      \vt{a^\prime}{b^\prime}(M\cdot \tau) = \vt{b}{a}(-\tau^{-1}) = \det(\tau)^{1/2}\e\left( \frac{a\Transpose{b}}{4} \right) \vt{a}{b}(\tau), \label{eq:U action}
		      \end{equation}
		      where the branch of \(\det(\tau)^{1/2}\) is assigned
		      so that
		      \(\det(iI_4)^{1/2} = 1. \)
		      In particular, 		      for \(v \in \mathbb{B}_3\),
		      \begin{equation}
			      \label{eq:inversion of characteristics}
			      \e \left( \frac{a\Transpose{b} }{4}\right) \vt{a}{b}(v) = \vt{bU}{aU}(v).
		      \end{equation}
		\item Let \(M=\jmath(g_{0,1})\) be the
		      image of \(g_{0,1}\) in \eqref{eq:half-turn-Mat} under the map \(\jmath\).
		      Then \(M\) is in \(\Sp(8,\Z)\), and
		      \[
			      \vt{a^\prime}{b^\prime}(M\cdot \tau) = \frac{1+i}{\sqrt{2}} \e\left( \phi_{a,b}(\jmath(g_{0,1})) \right)\chi(\jmath(g_{0,1}),\tau)^{1/2} \vt{a}{b}(\tau),
		      \]
		      where the branch of \(\chi(\jmath(g_{0,1}),\tau)^{1/2}\) is assigned
		      so that the real part of
		      \(
		      \chi(\jmath(g_{0,1}),i I_4)^{1/2}=\sqrt{-2-2i}
		      \) is positive.
		\item For \(M=\jmath(g_{0,1})^{-1}\),
		      \[
			      \vt{a^\prime}{b^\prime}(M\cdot \tau) = \frac{1-i}{\sqrt{2}}\e\left( \phi_{a,b}(\jmath(g_{0,1})^{-1}) \right) \chi(\jmath(g_{0,1})^{-1},\tau)^{1/2} \vt{a}{b}(\tau),
		      \]
		      where the branch of \(\chi(\jmath(g_{0,1})^{-1},\tau)^{1/2}\) is assigned
		      so that the real part of
		      \(
		      \chi(\jmath(g_{0,1})^{-1},i I_4)^{1/2}=\sqrt{-2+2i}
		      \) is positive.

		\item For \(M=\begin{pmatrix}
			      I_4 & I_4 \\
			      O_4 & I_4
		      \end{pmatrix},\ \begin{pmatrix}
			      I_4 & U   \\
			      O_4 & I_4
		      \end{pmatrix}\),
		      \[
			      \vt{a^\prime}{b^\prime}(M\cdot \tau) =\e\left( \phi_{a,b}(M) \right) \vt{a}{b}(\tau).
		      \]
	\end{enumerate}
\end{lemma}

\begin{proposition}
	If \(bU\Transpose{b} \not\equiv 0 \mod 4\) for \(b\in \mathbb{Z}^4\), then
	the theta constant \(\vartheta_{bU,b}(v)\) vanishes.
\end{proposition}
\begin{proof}
	By applying \eqref{eq:inversion of characteristics}
	to \(\vartheta_{bU,b}(v)\),
	we have
	\[
		\e \left( \frac{bU\Transpose{b} }{4}\right) \vt{bU}{b}(v) = \vt{bU}{b}(v),
	\]
	which yields the claim.
\end{proof}

\subsection{Relations Between Theta Constants and Branch Points}

We  choose \(x=(x_1,x_2,x_3)\)
in the neighborhood \(V\) of \(\dot x\) given in Subsection \ref{sec:Period maps}.
We fix it and set
\[
	v=\per|_{V}(x)\in \mathbb{B}_3, \quad
	\tau = \imath(v)\in \mathfrak{S}_4
\]
through this subsection.
We regard certain theta functions as holomorphic functions on \(\mathbb{C}^4\), and consider their pullbacks under the Abel-Jacobi-\(\Lambda\)  map \(\psi_\Lambda \colon C(x) \to A_\Lambda=H_-^0(C(x),\Omega^1)^\ast /\Lambda\).

\begin{definition}
	We define a map from the universal covering \(\widetilde{C}(x)\)
	of \(C(x)\) to \(\C^4\) by
	\[
		\psi \colon	\widetilde{C}(x)\ni \widetilde P=(P,\gamma) \mapsto
		\psi(P,\gamma) =
		\left( \int_{(1-\rho^2)\cdot \gamma}\varphi_1, \dots, \int_{(1-\rho^2)\cdot \gamma}\varphi_4 \right)
		\in  \mathbb{C}^4,
	\]
	where the base point of \(\widetilde{C}(x)\) is \(P_\infty\),
	\(\gamma\) is a path from \(P_\infty\) to \(P\), and
	\((1 - \rho^2)\cdot \gamma = \gamma - \rho^2\cdot\gamma\) is the path joining the start point \(P_\infty\) of \(\gamma\) to the end point \(P_\infty\) of the reverse path of \(\rho^2\cdot \gamma\).
	This map descends to  a map
	\[
		\psi_\Lambda \colon C(x) \ni P\mapsto  \psi(P)\in A_\Lambda=H^0_-(C(x),\Omega^1)^\ast /\Lambda(x)\simeq
		\mathbb{C}^4 / (\mathbb{Z}^4 \tau + \mathbb{Z}^4),
	\]
	which is called the Abel-Jacobi-\(\Lambda\) map.
\end{definition}
\begin{remark}
	The map \(\psi\) depends on a path \(\gamma\) connecting \(P_\infty\) to
	\(P\in C(x)\).
	For elements \((P,\gamma), (P,\gamma^\prime)\in \widetilde{C}(x) \),
	\(\gamma-\gamma^\prime\) represents an element of \(H_1(C(x),\mathbb{Z})\).
	Since \((1-\rho^2)\cdot H_1(C(x),\mathbb{Z}) \subset \Lambda(x)\),
	\( \psi(P,\gamma^\prime)\) is equal to \( \psi(P,\gamma)\)
	as elements of \(A_\Lambda\).
	Therefore, \(\psi_\Lambda\colon C(x) \to A_\Lambda\) is single valued.
\end{remark}
\begin{proposition}
	The images of the points \(P_j\ (j=0,x_1,x_2,x_3,1)\) under
	the Abel-Jacobi-\(\Lambda\)  map \(\psi_\Lambda\) are given as follows:
	\begin{align}
		& \psi_\Lambda(P_\infty) \equiv \psi_\Lambda(P_{1}) \equiv (0,0,0,0),                                 \\
		& \psi_\Lambda(P_0) \equiv \psi_\Lambda(P_{x_1}) \equiv \frac{1}{2}((1,0,0,0) \tau + (1,0,0,0)U),     \\
		& \psi_\Lambda(P_{x_2}) \equiv \psi_\Lambda(P_{x_3}) \equiv \frac{1}{2}((1,0,1,1) \tau + (1,0,1,1)U).
	\end{align}
\end{proposition}
\begin{proof}
	It is obvious that  \(\psi_\Lambda(P_\infty) \equiv (0,0,0,0)\).
	We may take \(L_1\), \(L_1\cdot L_2\), \dots, \(L_1\cdots L_5\), as the integration paths \(\gamma_j\) for
	\(\psi(P_j,\gamma_j)\) \((j=0,x_1,x_2,x_3,1)\).
	Note that their images under the map \((1-\rho^2)\) become cycles
	\(c_1\), \(c_1 + c_2\), \dots, 	\(c_1 + \cdots + c_5\), respectively.
	Since \(\Transpose{(A,B)} = T_U \Transpose{(C,\rho C)}=
	T_U\Transpose{(c_2,\ldots,c_5,\rho c_2,\ldots,\rho c_5)}\) in \eqref{eq:sigma minus to Lx}, we have
	\begin{align}
		\begin{pmatrix}
			C \\
			\rho C
		\end{pmatrix}\cdot (\varphi_1,\dots,\varphi_4) & = T_U^{-1}
		\begin{pmatrix}
			A \\
			B
		\end{pmatrix} \cdot
		(\phi_1,\dots,\phi_4) \tau_B^{-1} = T_U^{-1} \Pi \tau_B^{-1}  =
		T_U^{-1}\begin{pmatrix}
			\tau \\
			I_4
		\end{pmatrix},
	\end{align}
	where \(\cdot\) means the pairing of cycles and differential forms, which is given by integration.
	By using the relation \(c_1 = -c_4-c_5+\rho c_3 + \rho c_4\),
	we have the expressions of \(\psi_\Lambda(P_j)\).
\end{proof}
\begin{remark}
	Assume that \({\psi}^\ast \vt{a}{b}(\zeta,\tau) \colon \widetilde{C}(x) \to \mathbb{C}\) is not identically zero.
	Although this function is not single valued on \(C(x)\), the zero and
	its order of this multi-valued function on \(C(x)\) are well defined,
	since they are independent of the choice of paths used in the definition of \(\psi\)
	by the quasi periodicity
	\begin{equation}\label{eq:quasi-period}
		\vt{a}{b}(\zeta+n_1\tau+n_2,\tau)  = \e\left(
		\frac{1}{2}a\Transpose{n_2}- \frac{1}{2}n_1 \Transpose{b}
		-\frac{1}{2}n_1 \tau \Transpose{n_1} -n_1\Transpose{\zeta}  \right) \vt{a}{b}(\zeta,\tau)
	\end{equation}
	for  \(n_1,n_2 \in \mathbb{Z}^4\).
\end{remark}
By an argument similar to the proof of \cite[Proposition 4.2]{MT04}, we obtain the following.
\begin{proposition}
	Suppose that the pullback \(\psi^\ast  \vt{a}{b}(\zeta, \tau)\) is not identically zero. Regard it as a multi-valued function on \(C(x)\).
	Then the total number of its zero points is equal to eight with multiplicity.
\end{proposition}
\begin{proposition}\label{prop:rho orbits of zero are zero}
	If a point \(P \in C(x)\) is a zero of \(\psi^\ast \vt{bU}{b}\) for \(b\in \mathbb{Z}^4\), then
	the points \(\rho P, \rho^2 P\), and \(\rho^3 P\) are also zeros of
	this function.
\end{proposition}
\begin{proof}
	By computing the action of \(\rho\) on \(\psi(\widetilde{P})\)
	for \(\widetilde{P}=(P,\gamma)\in \widetilde{C}(x)\),
	we have \(\psi(\rho \widetilde{P}) = \psi(\widetilde{P})U\tau\).
	Applying \eqref{eq:tr-theta-function} to \(M=\jmath(-iI_4)\) and \(\zeta = \psi(\widetilde{P})\), we obtain
	\begin{equation}\label{eq:Mrho action}
		\vt{bU}{b}(\psi(\rho\widetilde{P}),\tau) = i^{\,bU\Transpose{b}} \e\left( \frac{1}{2}\psi(\widetilde{P})\tau^{-1}\Transpose{\psi(\widetilde{P})} \right) \vt{bU}{b}(\psi(\widetilde{P}),\tau),
	\end{equation}
	Hence \(\vt{bU}{b}(\psi(\widetilde{P}),\tau) = 0\) implies \(\vt{bU}{b}(\psi(\rho\widetilde{P}),\tau) = 0\).
\end{proof}
We define a function \(\vartheta_b(\widetilde{P})\) on  \(\widetilde{C}(x)\) by the pullback
\[
	\vartheta_b(\widetilde{P})=\psi^\ast \Big( \vt{bU}{b}(\zeta, \tau)\Big)=\vt{bU}{b}(\psi(\widetilde{P}), \tau)
\]
of \(\vt{bU}{b}(\zeta, \tau)\) for \(b \in \mathbb{Z}^4\)
under \(\psi\).
Then, by an argument analogous to \cite[Proposition 4.3]{MT04},
we can see the orders of \(\vartheta_b(\widetilde{P})\) at \(P_j\) \((j = 0, x_1, x_2, x_3, 1, \infty)\)  modulo \(4\).
\begin{proposition}
	Denote \(\psi_\Lambda(P_j) \equiv \frac{1}{2}(\xi_j \tau + \xi_j U)\) for \(j = 0, x_1, x_2, x_3, 1, \infty\), and set \(q = b + \xi_j U\). If \(\vartheta_b(\widetilde{P})\) is not identically zero, then the order
	of the zero of \(\vartheta_b(\widetilde{P})\) at \(P_j\)
	is congruent to \(q U \Transpose{q}\) modulo \(4\) for \(j = 0, x_1, x_2, x_3, 1\), and to \(-q U \Transpose{q}\) modulo \(4\) for \(j = \infty\).
	In particular, if \(b_3 + b_4\) is even, then the order is congruent to \(-q U \Transpose{q}\) modulo \(4\) at all six points.
\end{proposition}
This proposition yields the order of zero of \(\vartheta_b(\widetilde{P})\) at \(P_j\)
for each \(b \in \{0,1\}^4\) as in Table \ref{fig:orders of zero}.
The orders in Table \ref{fig:orders of zero} sum to \(16\), which exceeds the total number \(8\) of zeros; hence \(\vartheta_b(\widetilde{P})\) is identically zero for \(b = (1,0,0,1)\) and \((1,0,1,0)\).
\begin{table}[H]
	\centering
	\begin{tabular}{c|c|c|c|c}
		\(b\)         & \(P_0\), \(P_{x_1}\) & \(P_{x_2}\), \(P_{x_{3}}\) & \(P_1\) & \(P_\infty\) \\\hline
		\((0,0,0,0)\) & \(0\)                & \(2\)                      & \(0\)   & \(0\)        \\
		\((0,0,0,1)\) & \(1\)                & \(1\)                      & \(1\)   & \(3\)        \\
		\((0,0,1,0)\) & \(1\)                & \(1\)                      & \(1\)   & \(3\)        \\
		\((0,0,1,1)\) & \(2\)                & \(0\)                      & \(2\)   & \(2\)        \\
		\((0,1,0,0)\) & \(0\)                & \(2\)                      & \(0\)   & \(0\)        \\
		\((0,1,0,1)\) & \(1\)                & \(1\)                      & \(1\)   & \(3\)        \\
		\((0,1,1,0)\) & \(1\)                & \(1\)                      & \(1\)   & \(3\)        \\
		\((0,1,1,1)\) & \(2\)                & \(0\)                      & \(2\)   & \(2\)        \\
		\((1,0,0,0)\) & \(2\)                & \(0\)                      & \(0\)   & \(0\)        \\
		\((1,0,0,1)\) & \(3\)                & \(3\)                      & \(1\)   & \(3\)        \\
		\((1,0,1,0)\) & \(3\)                & \(3\)                      & \(1\)   & \(3\)        \\
		\((1,0,1,1)\) & \(0\)                & \(2\)                      & \(2\)   & \(2\)        \\
		\((1,1,0,0)\) & \(0\)                & \(2\)                      & \(2\)   & \(2\)        \\
		\((1,1,0,1)\) & \(1\)                & \(1\)                      & \(3\)   & \(1\)        \\
		\((1,1,1,0)\) & \(1\)                & \(1\)                      & \(3\)   & \(1\)        \\
		\((1,1,1,1)\) & \(2\)                & \(0\)                      & \(0\)   & \(0\)        \\
	\end{tabular}
	\caption{The order of zero of \(\vartheta_b(\widetilde{P_j})\)}
	\label{fig:orders of zero}
\end{table}

\begin{definition}
	For \(j = 0,1,2,3\),
	define \(\vartheta_j(\zeta,\tau(v)) = \vartheta_{\nu_j}(\zeta,\tau(v))\) and \(\vartheta_j(\widetilde{P})=\vartheta_{\nu_j}(\widetilde{P})\), where
	\begin{align}
		\nu_0 & = (0000,0000), & \nu_1 & = (1000,0100), & \nu_2 & = (0100,1000), & \nu_3 & = (1111,1111).
	\end{align}
	Let \(\vartheta_j(v)\)
	denote the theta constant \(\vartheta_j(0,\tau(v))\).
	Furthermore, set \(\Theta_{jk}(P) = \vartheta_j(\widetilde P)/\vartheta_k(\widetilde P)\) for \((j,k) = (0,1)\) and \((2,3)\).
	Though each function \(\vartheta_j(\widetilde P)\) is not single valued as a function
	on \(C(x)\),
	both \(\Theta_{01}(P)\) and \(\Theta_{23}(P)\) are single-valued functions on \(C(x)\) by the quasi periodicity of the theta function.
\end{definition}
\begin{proposition}\label{prop:R_jk is meromorphic}
	The functions \(\Theta_{01}\) and \(\Theta_{23}\) on \(C(x)\) are meromorphic functions with respect to \(P=(z,w)\in C(x)\). Furthermore, there exist constants \(s_{jk},t_{jk}\in \mathbb{C}\) and \(C_{jk} \neq 0\)
	independent of \(z\) such that
	\begin{equation}
		\label{eq:mero-h01-h23}
		\Theta_{01}(P) = C_{01} \frac{z-s_{01}}{z-t_{01}}, \quad
		\Theta_{23}(P) = C_{23} \frac{z-s_{23}}{z-t_{23}}.
	\end{equation}
	In particular,
	\begin{equation}
		\label{eq:theta-h01-h23}
		\Theta_{01}(P_0) = \frac{\vartheta_1(v)}{\vartheta_0(v)},\quad \Theta_{23}(P_{x_3}) = \frac{\vartheta_3(v)}{\vartheta_2(v)}.
	\end{equation}
\end{proposition}
\begin{proof}
	From Table \ref{fig:orders of zero}, the orders of \(\vartheta_0(\tilde{P})\) and \(\vartheta_1(\tilde{P})\) at \(P = P_{x_2}\) are congruent to \(2\) modulo \(4\). Since \(\vartheta_0(\tilde{P})\) and \(\vartheta_1(\tilde{P})\) are not identically zero, and each has exactly eight zeros, the orders at \(P_{x_2}\) must be equal to \(2\). Thus, each total order of vanishing of
	\(\vartheta_0(\tilde{P})\) and \(\vartheta_1(\tilde{P})\)
	at \(P_{x_2}\) and \(P_{x_3}\) is four.
	Consequently, the remaining zeros of \(\vartheta_0(\tilde{P})\) and \(\vartheta_1(\tilde{P})\) are four unramified points. Since these remaining zeros have a common \(z\)-coordinate,  we have
	\[
		\Theta_{01}(P) = C_{01} \frac{z - s_{01}}{z - t_{01}}.
	\]
	Since \(\psi(P_0) \equiv \frac{1}{2}(1,0,0,0)\tau + \frac{1}{2}(0,1,0,0)\), it follows that
	\begin{align}
		\vartheta_0(P_0,\tau) & = \vt{0000}{0000}(\psi(P_0), \tau) = \e\left( -\frac{1}{8}\tau_{11} \right) \vt{1000}{0100}(0, \tau) = \e\left( -\frac{1}{8}\tau_{11} \right)\vartheta_1(v), \\
		\vartheta_1(P_0,\tau) & = \vt{1000}{0100}(\psi(P_0), \tau) = \e\left( -\frac{1}{8}\tau_{11} \right) \vt{2000}{0200}(0, \tau) = \e\left( -\frac{1}{8}\tau_{11} \right)\vartheta_0(v).
	\end{align}
	Therefore, we obtain
	\[
		\Theta_{01}(P_0) = \frac{\vartheta_1(v)}{\vartheta_0(v)}.
	\]
	We have	the claim for \(\Theta_{23}\) by applying a similar argument to \(\vartheta_2(\tilde{P})\) and \(\vartheta_3(\tilde{P})\).
\end{proof}

\begin{proposition}\label{prop:relation each branch points}
	The following equalities hold:
	\begin{align}
		& \Theta_{01}(P_0) + \Theta_{01}(P_{x_1}) = 0,     &  & \Theta_{01}(P_1) + \Theta_{01}(P_\infty) = 0, \\
		& \Theta_{23}(P_{x_2}) + \Theta_{23}(P_{x_3}) = 0, &  & \Theta_{23}(P_1) + \Theta_{23}(P_\infty) = 0.
	\end{align}
\end{proposition}
\begin{proof}
	To show 	the first equality, we consider the value
	\(\psi(\widetilde{P}_{x_1})-\psi(\widetilde{P}_0)\),
	where \(\widetilde{P}_{x_1}=(P_{x_1},L_1+L_2)\) and \(\widetilde{P}_0=(P_0,L_1)\).
	Since
	\[
		\psi(P_{x_1},L_1+L_2)-\psi(P_0,L_1)=\int_{(1-\rho^2)\cdot L_2}\varphi,
		\quad (1-\rho^2)\cdot L_2=c_2 = -B_1 + A_3,
	\]
	we have
	\(\psi(P_{x_1})-\psi(P_0)=(0,0,1,0)\tau + (-1,0,0,0)\).
	By the quasi periodicity \eqref{eq:quasi-period}, we obtain the first equality.
	We can similarly show the others.
\end{proof}

\begin{corollary}\label{cor:relation between s and t}
	The constants \(s_{01}, t_{01}, s_{23}\), and \(t_{23}\)
	in \eqref{eq:mero-h01-h23}
	satisfy
	\begin{align}
		& s_{01} + t_{01} = 2, & \quad & s_{01} t_{01} = x_1,                 \\
		& s_{23} + t_{23} = 2, & \quad & s_{23} t_{23} = x_2 + x_3 - x_2 x_3.
	\end{align}
	Furthermore, the constants \(C_{01}\) and \(C_{23}\) in \eqref{eq:mero-h01-h23}
	are given by
	\[
		C_{01} = \frac{\vartheta_0(v)}{\vartheta_1(v)}, \quad C_{23} = \frac{\vartheta_2(v)}{\vartheta_3(v)}.
	\]
\end{corollary}
\begin{proof}
	Since the equalities in Proposition \ref{prop:relation each branch points} are
	equivalent to
	\begin{align}
		& \frac{x_1 - s_{01}}{x_1 - t_{01}} + \frac{s_{01}}{t_{01}} = 0,             & \quad & \frac{1 - s_{01}}{1 - t_{01}} + 1 = 0, \\
		& \frac{x_2 - s_{23}}{x_2 - t_{23}} + \frac{x_3 - s_{23}}{x_3 - t_{23}} = 0, & \quad & \frac{1 - s_{23}}{1 - t_{23}} + 1 = 0,
	\end{align}
	the assertions follow.
	We can determine the constants \(C_{01}\) and \(C_{23}\)
	by computing  \(\lim\limits_{P\to P_\infty} \Theta_{01}(P)\) and
	\(\lim\limits_{P\to P_\infty} \Theta_{23}(P)\), respectively.
\end{proof}
\begin{proposition}
	\label{prop:x1,(x2-x3)/(1-x3)}
	The theta constants \(\vartheta_0(v),\vartheta_1(v),\vartheta_2(v),\vartheta_3(v)\) satisfy the following equalities:
	\begin{align}
		& \frac{4\vartheta_0(v)^2\vartheta_1(v)^2}{(\vartheta_0(v)^2+\vartheta_1(v)^2)^2} = x_1, &  & \frac{4\vartheta_2(v)^2\vartheta_3(v)^2}{(\vartheta_2(v)^2+\vartheta_3(v)^2)^2} = \frac{x_2-x_3}{1-x_3}.
	\end{align}
\end{proposition}
\begin{proof}
	From Proposition \ref{prop:R_jk is meromorphic} and Corollary \ref{cor:relation between s and t}, we have
	\[
		\Theta_{01}(P_0) = \frac{\vartheta_1(v)}{\vartheta_0(v)},\quad \Theta_{01}(P_\infty) = \frac{\vartheta_0(v)}{\vartheta_1(v)}.
	\]
	Consequently, we obtain
	\[
		\frac{\Theta_{01}(P_0)}{\Theta_{01}(P_\infty)} = \frac{s_{01}}{t_{01}} = \frac{\vartheta_1(v)^2}{\vartheta_0(v)^2}.
	\]
	Thus, it follows that
	\begin{align}
		\left( 1+\frac{\vartheta_0(v)^2}{\vartheta_1(v)^2} \right)\left( 1+\frac{\vartheta_1(v)^2}{\vartheta_0(v)^2} \right) & =\left( 1+\frac{s_{01}}{t_{01}} \right) \left( 1+\frac{t_{01}}{s_{01}} \right) = \frac{(s_{01}+t_{01})^2}{s_{01}t_{01}} = \frac{4}{x_1},
	\end{align}
	which yields the first equality.
	By using the equalities
	\[
		\Theta_{23}(P_{x_3}) = \frac{\vartheta_3(v)}{\vartheta_2(v)},\quad \Theta_{23}(P_{\infty}) = \frac{\vartheta_2(v)}{\vartheta_3(v)},
	\]
	we similarly obtain the second equality.
\end{proof}

As shown in Proposition \ref{prop:x1,(x2-x3)/(1-x3)},
the branch point \(x_1\) is expressed in terms of theta constants.
In the next subsection, we express \(x_2\) and \(x_3\)
in terms of theta constants.

\subsection{Actions of Certain Rational Symplectic Elements on Theta Constants}
In this subsection, we study the actions of \(g_{j,k}\) in \eqref{eq:half-turn-Mat} on \(\vt{a}{b}(v)\). We prepare a lemma.
\begin{lemma}[{\cite[Lemma 5]{MMN07}}]
	Pairs of theta constants on the Siegel upper half-space \(\mathfrak{S}_2\) of degree 2 satisfy
	\[
		\begin{pmatrix}
			\vartheta_{a,b}(\tau_2+\Delta) \\
			\vartheta_{a,b+e}(\tau_2+\Delta)
		\end{pmatrix} = \nabla(a)
		\begin{pmatrix}
			\vartheta_{a,a\Delta+b}(\tau_2) \\
			\vartheta_{a,a\Delta+b+e}(\tau_2)
		\end{pmatrix},
	\]
	where \(\tau_2\in \mathfrak{S}_2\), \(e = (1,1)\), \(\Delta = \cfrac{1}{2}\begin{pmatrix}
		1  & -1 \\
		-1 & 1
	\end{pmatrix}\), and
	\begin{align}
		& \nabla(a) = \frac{c_1(a)}{2}
		\begin{pmatrix}
			1+i         & (1-i) c_2^{-1}(a) \\
			(1-i)c_2(a) & 1+i
		\end{pmatrix},                                                                                                       \\
		& c_1(a) = \exp\left( -\frac{\pi i a\Delta \Transpose{a}}{4} \right),\ c_2(a) = \exp \left( \frac{\pi i a \Transpose{e}}{2} \right).
	\end{align}
\end{lemma}
\begin{proposition} \label{prop:g13 action}
	By the action of \(g_{1,3} \in \U(U,\Q(i))\) in
	\eqref{eq:half-turn-Mat}, we have
	\begin{align}
		\begin{pmatrix}
			\vt{a}{b}(g_{1,3}v) \\
			\vt{a+e_3 + e_4}{b+e_3 + e_4}(g_{1,3}v)
		\end{pmatrix}
		= & \chi(\jmath(g_{1,3}),\tau)^{1/2} \e\left( \frac{(a_4-a_3)(b_4-b_3)}{8} \right) \frac{1+i}{2} \label{eq:g13 action} \\
		& \times
		\begin{pmatrix}
			\e\left( -\frac{a_3+a_4}{4} \right) & 1                                   \\
			1                                   & -\e\left( \frac{a_3+a_4}{4} \right)
		\end{pmatrix}
		\begin{pmatrix}
			\vt{c}{d+e_3 + e_4}(v) \\
			\vt{c+e_3 + e_4}{d}(v)
		\end{pmatrix},
	\end{align}
	where \(\jmath(g_{1,3})\) is
	the image of \(g_{1,3}\) under the map \(\jmath\) in
	\eqref{map:def of emb between auto grp} in Proposition \ref{prop:modularembed},
	\((c,d) = \jmath(g_{1,3})^{-1}\cdot (a,b)\), \(e_3 = (0,0,1,0)\), \(e_4 = (0,0,0,1)\), and the branch of the square root is chosen so that its value at \(\tau = iI_4 \in \mathfrak{S}_4\) is
	\[
		\chi(\jmath(g_{1,3}),iI_4)^{1/2} = \sqrt{-i} = \frac{1-i}{\sqrt{2}}.
	\]
\end{proposition}
\begin{proof}
	Since
	\[
		\jmath(g_{1,3})=
		\begin{pmatrix}
			I_2 &             &     &             \\
			    & I_2 -\Delta &     & \Delta      \\
			    &             & I_2 &             \\
			    & -\Delta     &     & I_2 -\Delta
		\end{pmatrix},
	\]
	we obtain the result by an argument similar to that in
	\cite[Proposition 3]{MMN07}.
\end{proof}
\begin{proposition}\label{prop:g12 action}
	For \(a,b \in \mathbb{Z}^4\),
	the point \(\tau = \tau(v) \in \mathfrak{S}_4\) and \(g_{j,k}\) in \eqref{eq:half-turn-Mat}, set \((c,d) = (a,b)\jmath(g_{1,2})\).
	Then, the following equality holds:
	\begin{align}
		\vt{a}{b}(g_{1,2}v)= & \chi(\jmath(g_{1,2}),\imath(v))^{1/2}E(g_{1,2})_{a,b} \label{eq:g12 action}                                                           \\
		& \times\left\{\! E_1(g_{1,2})_{a,b} \vt{c\!+\!e_4}{d\!+\!e_3}(v) \!+\! E_1(g_{1,2})_{a,b}^{-1}\vt{c\!+\!e_3}{d\!+\!e_4}(v)\! \right\},
	\end{align}
	where we set \(E_1(g_{1,2})_{a,b} = \e\left( (a_4-b_4)/8 \right)\),
	\begin{align}
		E(g_{1,2})_{a,b} = & \frac{1+i}{2} \e\left( \frac{-a_3+b_3}{8} \right) \e\left(\frac{(a_3-b_4)(a_4+b_3)}{8} \right) \\
		& \times
		\e\left( \frac{a_4 b_4}{4}\right) \e\left(-\frac{(a_3+a_4-b_3-b_4)(a_3+a_4+b_3+b_4)}{16} \right),
	\end{align}
	and the branch of \(\chi(\jmath(g_{1,2}),\imath(v))^{1/2}\) is assigned so that
	\[
		\chi(\jmath(g_{1,2}),i I_4)^{1/2}=\sqrt{-i}=\dfrac{1-i}{\sqrt{2}}.
	\]
\end{proposition}
\begin{proof}
	Note that \[
		g_{1,2} = g_{2,3}^{-1} g_{1,3} g_{2,3}.
	\]
	We can show the formula by Lemma \ref{lem:monodromy action and U action} (1), (2), Proposition \ref{prop:g13 action} and
	\[
		\chi(LM,\tau) = \chi(L,M\cdot \tau)\chi(M,\tau)
	\]
	for \(L,M \in \Sp(8,\mathbb{Q})\) with considering branches of square roots.
\end{proof}

\begin{corollary}
	\label{cor:g-actions}
	For \(j=0,1,2,3\), write \(\nu_j = (a, b)\). Then the equalities
	\begin{align}\label{eq:g12-action-vt}
		\vartheta_j(g_{1,2} v) & = \chi(\jmath(g_{1,2}),\imath(v))^{1/2}(1+i) E(g_{1,2})_{\nu_j}^\prime  \vt{a+e_3}{b+e_4}(v),
		\\ \label{eq:g13-action-vt}
		\vartheta_j(g_{1,3} v) & = \chi(\jmath(g_{1,3}),\imath(v))^{1/2}(1+i) \vt{a+e_3 + e_4}{b}(v),
	\end{align}
	hold, where
	\begin{equation}\label{eq:Eg12}
		E(g_{1,2})_{\nu_j}^\prime=\left\{
		\begin{array}{rcl}
			1  & \textrm{ for } & j=0,1,2, \\
			-i & \textrm{ for } & j=3.
		\end{array}\right.
	\end{equation}
\end{corollary}
\begin{proof}
	We prove \eqref{eq:g13-action-vt}; the equality \eqref{eq:g12-action-vt}
	follows in the same way from \eqref{eq:g12 action}.
	Since \(aU=b\), \(bU=a\) and \((e_3+e_4)U=e_3+e_4\), the equality
	\(\jmath(g_{1,3})^{-1} \cdot \nu_j=\nu_j\) holds.
	Reading off the first row of \eqref{eq:g13 action}, in which
	\((c,d)=(a,b)\) by this invariance,
	we obtain
	\begin{align}
		\vartheta_j(g_{1,3}v)= {} & \chi(\jmath(g_{1,3}),\imath(v))^{1/2}
		\e\!\left(\frac{(a_4-a_3)(b_4-b_3)}{8}\right)\frac{1+i}{2}\label{eq:g13-first-row}\\
		                          & \times\left[\e\!\left(-\frac{a_3+a_4}{4}\right)\vt{a}{b+e_3+e_4}(v)
		+\vt{a+e_3+e_4}{b}(v)\right].
	\end{align}
	Applying \eqref{eq:inversion of characteristics} to the first theta constant
	and using \(aU=b\) and \((b+e_3+e_4)U=a+e_3+e_4\), we obtain
	\[
		\vt{a}{b+e_3+e_4}(v)
		=\e\!\left(-\frac{a\Transpose{(b+e_3+e_4)}}{4}\right)
		\vt{a+e_3+e_4}{b}(v),
	\]
	so that the last factor in \eqref{eq:g13-first-row} becomes
	\[
		\biggl(\e\Bigl(-\frac{a_3\!+\!a_4}{4}\Bigr)
		\e\Bigl(-\frac{a\Transpose{(b+e_3+e_4)}}{4}\Bigr)\!+\!1\biggr)
		\vt{a\!+\!e_3\!+\!e_4}{b}(v)
		=2\,\vt{a\!+\!e_3\!+\!e_4}{b}(v),
	\]
	since a direct computation shows that
	\(\e\!\left(-\frac{a_3+a_4}{4}\right)
	\e\!\left(-\frac{a\Transpose{(b+e_3+e_4)}}{4}\right)=1\) for each \(j\).
	Here note that each factor is \(-1\) when \(j=3\).
	Since \(\e\!\left((a_4-a_3)(b_4-b_3)/8\right)=1\) for each \(j\),
	the coefficient reduces to \(1+i\); this proves \eqref{eq:g13-action-vt}.

	The equality \eqref{eq:g12-action-vt} follows from \eqref{eq:g12 action} by
	the same reduction, \eqref{eq:inversion of characteristics} and the
	quasi-periodicity of theta constants in the characteristics; the resulting
	root of unity is recorded by the factor \(E(g_{1,2})^\prime_{\nu_j}\) of
	\eqref{eq:Eg12}, which becomes \(-i\) for \(j=3\) and \(1\) for \(j=0,1,2\).
\end{proof}
\begin{corollary}\label{cor:x2, x3, (x3-x1)/(1-x1), (x2-x1)/(1-x1)}
	Set  	\begin{align}
		\nu_4 & = (0010,0001), & \nu_5 & = (1010,0101), & \nu_{6}  & = (0110,1001), & \nu_{7}  & = (1101,1110), \\
		\nu_8 & = (0011,0000), & \nu_9 & = (1011,0100), & \nu_{10} & = (0111,1000), & \nu_{11} & = (1100,1111),
	\end{align}
	and  define \(\vartheta_j(v)\) as \(\vartheta_{\nu_j}(v)\)
	for \(v\in \mathbb{B}_3\).
	Then we have
	\begin{align}
		& \frac{4\vartheta_{4}(v)^2\vartheta_{5}(v)^2}{(\vartheta_{4}(v)^2+\vartheta_{5}(v)^2)^2} = x_2, &  & \frac{4\vartheta_{6}(v)^2\vartheta_{7}(v)^2}{(\vartheta_{6}(v)^2+\vartheta_{7}(v)^2)^2} = \frac{x_3-x_1}{1-x_1},             \\
		& \frac{4\vartheta_{8}(v)^2\vartheta_{9}(v)^2}{(\vartheta_{8}(v)^2+\vartheta_{9}(v)^2)^2} = x_3, &  & \frac{4\vartheta_{{10}}(v)^2\vartheta_{{11}}(v)^2}{(\vartheta_{{10}}(v)^2+\vartheta_{{11}}(v)^2)^2} = \frac{x_2-x_1}{1-x_1}.
	\end{align}
\end{corollary}

\begin{proof}
	We show the first equality. By the construction of \(g_{1,2}\),
	we have
	\[
		\frac{4\vartheta_0(g_{1,2}v)^2 \vartheta_1(g_{1,2}v)^2}{(\vartheta_0(g_{1,2}v)^2 +  \vartheta_1(g_{1,2}v)^2)^2} = x_2.
	\]
	Using \eqref{eq:g12-action-vt}, we obtain
	\begin{align}
		\vartheta_0(g_{1,2} v) & = \chi(\jmath(g_{1,2}),\imath(v))^{1/2}(1+i) \vartheta_4(v), \\
		\vartheta_1(g_{1,2} v) & = \chi(\jmath(g_{1,2}),\imath(v))^{1/2}(1+i) \vartheta_5(v).
	\end{align}
	Therefore, the first equality follows.
	By acting \(g_{1,2}\) on the second equality in Proposition
	\ref{prop:x1,(x2-x3)/(1-x3)},
	we have
	\[
		-\frac{4\vartheta_{6}(v)^2\vartheta_{7}(v)^2}{(\vartheta_{6}(v)^2-\vartheta_{7}(v)^2)^2} = \frac{x_1-x_3}{1-x_3}
	\]
	by Corollary \ref{cor:g-actions}.
	It yields the second equality.
	We can similarly show the others.
\end{proof}

We conclude this subsection by the following theorem.
\begin{theorem}	For \(v \in \overline{\Gamma\backslash \mathbb{B}_3}\), we define
	\[
		x_j(v) = \frac{4\vartheta_{4j-4}(v)^2\vartheta_{4j-3}(v)^2}{\left(\vartheta_{4j-4}(v)^2 + \vartheta_{4j-3}(v)^2\right)^2} \quad (j = 1,2,3).
	\]
	These functions are invariant under the action of the monodromy group
	\(\Gamma\), and the image of \((x_1(v), x_2(v), x_3(v)) \in \C^3\subset \Ps^3\) under the period map \(\per \colon \mathbb{P}^3 \to \overline{\Gamma\backslash \mathbb{B}_3}\) coincides with \(v \in \overline{\Gamma\backslash \mathbb{B}_3}\). That is, the inverse of the isomorphism \(\per\) is given by
	\[
		{\per}^{-1}(v) = (x_1(v), x_2(v), x_3(v)).
	\]
\end{theorem}

\subsection{Thomae-Type Formulas}
In this section, we present several formulas describing relationships
between period integrals and theta constants on
\(\mathbb{B}_3\). Such formulas are called Thomae-type formulas.
\begin{theorem}\label{thm:Thomae 1}
	We take an element \((x_1,x_2,x_3)\in X\) and consider the period integrals
	\[
		v = \Transpose{\left( \int_{B_1} \frac{dz}{w},\dots, \int_{B_4} \frac{dz}{w} \right)}\in \mathcal{B}
	\]
	associated with \((x_1,x_2,x_3)\). Then there exists a constant \(\kappa\) such that the following equalities hold:
	\begin{align}
		& (\vartheta_0(v)^2 + \vartheta_1(v)^2)^2 = \kappa (\Transpose{v} U v)^2,          &  & (\vartheta_2(v)^2+\vartheta_3(v)^2)^2 = \kappa (1-x_3)(\Transpose{v} U v)^2,                  \\
		& (\vartheta_4(v)^2+\vartheta_5(v)^2)^2 = \frac{1}{4}\kappa (\Transpose{v} U v)^2, &  & (\vartheta_6(v)^2+\vartheta_7(v)^2)^2 = \frac{1}{4}\kappa (1-x_1)(\Transpose{v} U v)^2,       \\
		& (\vartheta_8(v)^2+\vartheta_9(v)^2)^2 = \frac{1}{4}\kappa(\Transpose{v} U v)^2,  &  & (\vartheta_{10}(v)^2+\vartheta_{11}(v)^2)^2 = \frac{1}{4}\kappa (1-x_1)(\Transpose{v} U v)^2.
	\end{align}
	Furthermore, the constant \(\kappa\) is given by \(\left( (16\pi)^2 \varGamma(3/4)^8 \right)^{-1}\).
\end{theorem}
\begin{proof} We show
	\((\vartheta_0(v)^2 + \vartheta_1(v)^2)^2 = \kappa (\Transpose{v} U v)^2\).
	We firstly show that
	\(\dfrac{(\vartheta_0(v)^2 + \vartheta_1(v)^2)^2}{(\Transpose{v} U v)^2}\)
	is a holomorphic function on \(\Gamma\backslash\mathbb{B}_3\).
	Since \(\Transpose{v} U v\ne 0\) for \(v\in \mathcal{B}\)
	by Lemma \ref{lem:tvUv-nonzero},
	this function is holomorphic on \(\mathbb{B}_3\).
	We check that the numerator \((\vartheta_0(v)^2 + \vartheta_1(v)^2)^2\)
	and the denominator \((\Transpose{v} U v)^2\) transform
	in the same way under the action of \(g\in \Gamma\).
	Since the monodromy group \(\Gamma\) is generated by
	the elements \(g_{j,k}^2\) \((0\le j<k\le 4)\) in \(\U(U,\Z)\),
	it suffices to consider \(g = g_{j,k}^2\).
	Since \(\jmath(g)\) is in \(\Sp(8,\Z)\), we apply formula \eqref{eq:tr-theta} in Lemma \ref{lem:Igusa-formula} to \(\jmath(g)\) and \(\nu_l\) for \(l=0,1\) using the relation \eqref{eq:modular embedding}.
	We have
	\[
		\vartheta_l(gv)^2 = \kappa(\jmath(g))^2 \e\left( 2\phi_{\nu_l}(\jmath(g)) \right) \chi(\jmath(g),\imath(v)) \vartheta_l(v)^2 \quad (l = 0,1),
	\]
	where we use that \(\jmath(g)\) fixes the characteristics \(\nu_0\) and \(\nu_1\) modulo \(2\), and that the square of \(\kappa(\jmath(g))\) is independent of the characteristic.
	By a direct computation, we have
	\(\e\left( \phi_{\nu_0}(\jmath(g)) \right)^2 = \e\left( \phi_{\nu_1}(\jmath(g)) \right)^2 = 1\),
	and hence
	\[
		\vartheta_0(gv)^2 + \vartheta_1(gv)^2 = \kappa(\jmath(g))^2 \chi(\jmath(g),\imath(v)) (\vartheta_0(v)^2 + \vartheta_1(v)^2).
	\]
	Furthermore, since \(\kappa(\jmath(g))^4 = 1\),
	taking the square of this equality, we obtain
	\begin{equation}\label{eq:numerator invariance}
		(\vartheta_0(gv)^2 + \vartheta_1(gv)^2)^2 = \chi(\jmath(g),\imath(v))^2 (\vartheta_0(v)^2 + \vartheta_1(v)^2)^2.
	\end{equation}
	Combining \eqref{eq:numerator invariance} with the equality
	\eqref{eq:auto-factor} in Proposition \ref{prop:eval Psi}
	and \(\det(g) = \pm 1\), we have
	\begin{align}
		& \dfrac{(\vartheta_0(g v)^2 + \vartheta_1(g v)^2)^2}
		{(\Transpose{(g v)} U(g v))^2}
		=
		\chi(\jmath(g),\imath(v))^2\dfrac{(\vartheta_0(v)^2 + \vartheta_1(v)^2)^2}
		{(\Transpose{(g v)} U(g v))^2}                          \\
		= & \Big(\frac{\Transpose{(g v)}U(g v)}
			{\det(g )\Transpose{v}Uv}\Big)^2
		\dfrac{(\vartheta_0(v)^2 + \vartheta_1(v)^2)^2}
		{(\Transpose{(g v)} U(g  v))^2}=\dfrac{(\vartheta_0(v)^2 + \vartheta_1(v)^2)^2}
		{(\Transpose{v} Uv)^2}.
	\end{align}

	We secondly show that its pullback
	\[
		f(x_1,x_2,x_3)=
		{\per}^\ast  \left( \frac{\left( \vartheta_{0}(v)^2+\vartheta_{1}(v)^2 \right)^2}{(\Transpose{v}Uv)^2} \right)
	\]
	under the period map
	\(\per:\Ps^3\to \overline{\Gamma \backslash \mathbb{B}_3}\) becomes a
	constant.
	By the Satake-Baily-Borel compactification,	the following five points are added to \(\Gamma\backslash\mathbb{B}_3\)
	\begin{align}
		v_{16} & = \Transpose{(1,0,0,0)},                         & v_{26} & = g_{0,1}^{-1} v_{16} = \Transpose{(1,-1-i,-1+i,0)}, \\
		v_{36} & = g_{1,2}^{-1} v_{26} =\Transpose{(1,-1-i,i,i)}, & v_{46} & = g_{2,3}^{-1} v_{36} =\Transpose{(1,-1-i,i,1)},     \\
		v_{56} & = g_{3,4}^{-1} v_{46} =\Transpose{(0,1,0,0)}.    &        &
	\end{align}
	The points \(v_{16}\) and \(v_{56}\) correspond to
	\((1,1,1)\) and \((0,0,0)\) in \(\C^3(\subset \Ps^3)\),
	and \(v_{26}\), \(v_{36}\), \(v_{46}\) correspond to
	\([1,0,0,0]\), \([0,1,0,0]\), \([0,0,1,0]\) in \(\Ps^3\), respectively.
	Since \(\Transpose{v}Uv\ne0\) on \(\mathbb{B}_3\)
	by Lemma \ref{lem:tvUv-nonzero},
	\(f(x_1,x_2,x_3)\) is holomorphic on \(\Ps^3\) minus these five points.
	The Hartogs theorem yields that \(f(x_1,x_2,x_3)\) is holomorphic on \(\Ps^3\),
	and hence it is constant.

	We thirdly determine this constant \(\kappa=f(x_1,x_2,x_3)\)
	by setting  \(x_1=x_2=x_3\) and restricting \(x_1\) to the open interval \((0,1)\).
	By
	\begin{align}
		v_1 & = 2\sqrt{2}\pi F_D\left( \frac{1}{4},\frac{1}{4},\frac{1}{4},\frac{1}{4},1;x_1,x_2,x_3 \right), \\
		v_2 & = -4\pi F_D\left( \frac{1}{4},\frac{1}{4},\frac{1}{4},\frac{1}{4},1;1-x_1,1-x_2,1-x_3 \right),
	\end{align}
	given in \eqref{eq:v1-HGS} and \eqref{eq:v2-HGS}, we have
	\(v(x_1,x_1,x_1) = \per(x_1,x_1,x_1) = \Transpose{(v_1,v_2,0,0)}\) for \(x_1 \in (0,1)\subset \mathbb{C}\),
	and
	\[
		\imath(v(x_1,x_1,x_1)) = \diag(\tau_1, -\tau_1^{-1},i,i),\quad \tau_1 
		= \sqrt{2}i \dfrac{F(1-x_1)}{F(x_1)},
	\]
	where we set \(F(x_1) = F_D\left( \frac{1}{4},\frac{1}{4},\frac{1}{4},
	\frac{1}{4},1;x_1,x_1,x_1\right)\), which reduces to
	\(F\left( \frac{1}{4}, \frac{3}{4},1;x_1 \right)\)
	by the Euler type integral representation \eqref{eq:Euer-int-rep}.
	Thus  \(\kappa\) is equal to
	\begin{align}
		f(x_1,x_1,x_1)
		& = \frac{(\vartheta_0(v)^2 + \vartheta_1(v)^2)^2}{(\Transpose{v}Uv)^2}
		= \left( \vartheta_{00}(i)^4 (-i\tau_1) \frac{\vartheta_{00}(\tau_1)^4 + \vartheta_{10}(\tau_1)^4}{-16\sqrt{2}\pi^2 F(x_1) F(1-x_1)} \right)^2                                                                          \\
		& = \frac{\vartheta_{00}(i)^8}{256 \pi^4} \frac{(\vartheta_{00}(\tau_1)^4 + \vartheta_{10}(\tau_1)^4)^2 }{F(x_1)^4}=\frac{\vartheta_{00}(i)^8}{256 \pi^4}=\frac{1}{(16\pi)^2 \varGamma\left( \frac{3}{4} \right)^{8}},
	\end{align}
	where we use Proposition \ref{prop:x1,(x2-x3)/(1-x3)}, the formula \eqref{prop:Borwein result 2} and
	\[
		\vartheta_{00}(-\tau_1^{-1})^2=(-i\tau_1)\vartheta_{00}(\tau_1)^2,\quad
		\vartheta_{01}(-\tau_1^{-1})^2=(-i\tau_1)\vartheta_{10}(\tau_1)^2,\quad
		\vartheta_{00}(i)=\frac{\pi^{1/4}}{\varGamma(\frac{3}{4})}.
	\]

	The equality
	\((\vartheta_2(v)^2+\vartheta_3(v)^2)^2 = \kappa (1-x_3)(\Transpose{v} U v)^2\)
	can be similarly shown.
	We can prove the rest
	by applying
	\eqref{eq:g12-action-vt} and \eqref{eq:g13-action-vt} in
	Corollary \ref{cor:g-actions}
	to the obtained equalities
	together with \eqref{eq:auto-factor} in Proposition \ref{prop:eval Psi}.
\end{proof}

\begin{corollary}\label{thm:Thomae 2}
	For \(x=(x_1,x_2,x_3)\in X\), we set \(\tilde{x}=(0,x,1)\in \C^5-\Diag\), and the period integrals \(v\in \mathcal{B}\)
	as in Theorem \ref{thm:Thomae 1}.
	We list equalities between polynomials of \(\vartheta_j(v)\) on \(\mathbb{B}_3\) and the products of polynomials \(\mathcal{P}_{n_1,n_2;n_3,n_4}(\tilde x)\) given in \eqref{eq:15polys} and  \(\Xi=\kappa^2(\Transpose{v}Uv)^4\) in Table \ref{tab:correspondence on Thomae}.
	\begin{table}[hbt]
		\begin{gather}
			-16\prod_{j=0}^3 \vartheta_{j}^2
			=\mathcal{P}_{0,1;2,3}(\tilde{x})\Xi,\\
			-4(\vartheta_{0}^2+ \vartheta_{1}^2 )^2\vartheta_{2}^2 \vartheta_{3}^2
			=\mathcal{P}_{0,4;2,3}(\tilde{x})\Xi,\quad
			4\vartheta_{0}^2 \vartheta_{1}^2 (\vartheta_{2}^2 + \vartheta_{3}^2)^2
			=\mathcal{P}_{0,1;3,4}(\tilde{x})\Xi,\\
			-4(\vartheta_{0}^2- \vartheta_{1}^2 )^2\vartheta_{2}^2 \vartheta_{3}^2
			=\mathcal{P}_{1,4;2,3}(\tilde{x})\Xi,\quad
			4\vartheta_{0}^2 \vartheta_{1}^2 (\vartheta_{2}^2 - \vartheta_{3}^2)^2
			=\mathcal{P}_{0,1;2,4}(\tilde{x})\Xi,\\
			256\prod_{j=4}^7 \vartheta_{j}^2
			=\mathcal{P}_{0,2;1,3}(\tilde{x})\Xi,\\
			64(\vartheta_{4}^2+ \vartheta_{5}^2 )^2\vartheta_{6}^2 \vartheta_{7}^2
			=\mathcal{P}_{0,4;1,3}(\tilde{x})\Xi,\quad
			64\vartheta_{4}^2 \vartheta_{5}^2 (\vartheta_{6}^2 + \vartheta_{7}^2)^2
			=\mathcal{P}_{0,2;1,4}(\tilde{x})\Xi,\\
			64(\vartheta_{4}^2- \vartheta_{5}^2 )^2\vartheta_{6}^2 \vartheta_{7}^2
			=\mathcal{P}_{2,4;1,3}(\tilde{x})\Xi,\quad
			64\vartheta_{4}^2 \vartheta_{5}^2 (\vartheta_{6}^2 - \vartheta_{7}^2)^2
			=\mathcal{P}_{0,2;3,4}(\tilde{x})\Xi,\\
			256\prod_{j=8}^{11} \vartheta_{j}^2
			=\mathcal{P}_{0,3;1,2}(\tilde{x})\Xi,\\
			64(\vartheta_{8}^2+ \vartheta_{9}^2 )^2\vartheta_{10}^2 \vartheta_{11}^2
			=\mathcal{P}_{0,4;1,2}(\tilde{x})\Xi,\quad
			64\vartheta_{8}^2 \vartheta_{9}^2 (\vartheta_{10}^2 + \vartheta_{11}^2)^2
			=\mathcal{P}_{0,3;1,4}(\tilde{x})\Xi,\\
			64(\vartheta_{8}^2- \vartheta_{9}^2 )^2\vartheta_{10}^2 \vartheta_{11}^2
			=\mathcal{P}_{3,4;1,2}(\tilde{x})\Xi,\quad
			64\vartheta_{8}^2 \vartheta_{9}^2 (\vartheta_{10}^2 - \vartheta_{11}^2)^2
			=\mathcal{P}_{0,3;2,4}(\tilde{x})\Xi.
		\end{gather}

		\caption{The relations between the theta constants and \(\mathcal{P}_{n_1,n_2;n_3,n_4}(\tilde{x})\).}
		\label{tab:correspondence on Thomae}
	\end{table}
\end{corollary}
\begin{proof}
	We show the first equality.
	By using the first line in Theorem \ref{thm:Thomae 1} and the relations
	\begin{align}
		& \frac{4\vartheta_0(v)^2\vartheta_1(v)^2}{(\vartheta_0(v)^2+\vartheta_1(v)^2)^2} = x_1, & \frac{4\vartheta_2(v)^2\vartheta_3(v)^2}{(\vartheta_2(v)^2+\vartheta_3(v)^2)^2} = \frac{x_2-x_3}{1-x_3},
	\end{align}
	in Proposition \ref{prop:x1,(x2-x3)/(1-x3)}, we have
	\[
		16 \vartheta_0(v)^2\vartheta_1(v)^2\vartheta_2(v)^2\vartheta_3(v)^2 = \kappa^2(\Transpose{v}Uv)^4 x_1 (x_2-x_3) =- \mathcal{P}_{0,1;2,3}(\tilde{x})\Xi.
	\]
	We can show the others by using Corollary \ref{cor:x2, x3, (x3-x1)/(1-x1), (x2-x1)/(1-x1)}.
\end{proof}

The following corollary follows immediately from Theorem \ref{thm:Thomae 1}.
\begin{corollary}\label{cor:Thomae 3}
	Under the setting of Theorem \ref{thm:Thomae 1}, the equalities
	\begin{equation}
		\left( \frac{\vartheta_0(v)^2+\vartheta_1(v)^2}{2} \right)^2 = \left( \vartheta_4(v)^2+\vartheta_5(v)^2 \right)^2 = \left( \vartheta_8(v)^2+\vartheta_9(v)^2 \right)^2 = \frac{\kappa}{4} (\Transpose{v}Uv)^2 \label{cor:Thomae 3-1}
	\end{equation}
	hold. 	In particular, we have
	\begin{equation}
		\label{eq:quadratic-relations}
		\begin{array}{llll}
			& \dfrac{\vartheta_0^2+\vartheta_1^2}{2} = \vartheta_4^2 + \vartheta_5^2 = \vartheta_8^2 + \vartheta_9^2,       &
			& \dfrac{\vartheta_2^2+\vartheta_3^2}{2} = \vartheta_6^2 - \vartheta_7^2 = \vartheta_8^2 - \vartheta_9^2,             \\
			& \dfrac{\vartheta_0^2-\vartheta_1^2}{2} = \vartheta_6^2 + \vartheta_7^2 = \vartheta_{10}^2 + \vartheta_{11}^2, &
			& \dfrac{\vartheta_2^2-\vartheta_3^2}{2} = \vartheta_4^2 - \vartheta_5^2 = \vartheta_{10}^2 - \vartheta_{11}^2,       \\
			& \dfrac{\vartheta_0^2+\vartheta_2^2}{2} = \vartheta_4^2 + \vartheta_6^2 = \vartheta_8^2 + \vartheta_{10}^2,    &   & \\
			& \vartheta_4^2 = \dfrac{\vartheta_0^2 + \vartheta_1^2 + \vartheta_2^2 - \vartheta_3^2}{4},                     &
			& \vartheta_6^2 = \dfrac{\vartheta_0^2 - \vartheta_1^2 + \vartheta_2^2 + \vartheta_3^2}{4},                           \\
			& \vartheta_8^2 = \dfrac{\vartheta_0^2 + \vartheta_1^2 + \vartheta_2^2 + \vartheta_3^2}{4},                     &
			& \vartheta_{10}^2 = \dfrac{\vartheta_0^2 - \vartheta_1^2 + \vartheta_2^2 - \vartheta_3^2}{4}.
		\end{array}
	\end{equation}
\end{corollary}
\begin{proof}
	The equalities \eqref{cor:Thomae 3-1} hold by Theorem \ref{thm:Thomae 1}.
	By Corollary \ref{thm:Thomae 2}, we have
	\begin{equation}
		\begin{array}{lll}
			& \big( \dfrac{\vartheta_0^2+\vartheta_1^2}{2} \big)^2 = (\vartheta_4^2 + \vartheta_5^2)^2 = (\vartheta_8^2 + \vartheta_9^2)^2,       & \big( \dfrac{\vartheta_2^2+\vartheta_3^2}{2} \big)^2 = (\vartheta_6^2 - \vartheta_7^2)^2 = (\vartheta_8^2 - \vartheta_9^2)^2,       \\
			& \big( \dfrac{\vartheta_0^2-\vartheta_1^2}{2} \big)^2 = (\vartheta_6^2 + \vartheta_7^2)^2 = (\vartheta_{10}^2 + \vartheta_{11}^2)^2, & \big( \dfrac{\vartheta_2^2-\vartheta_3^2}{2} \big)^2 = (\vartheta_4^2 - \vartheta_5^2)^2 = (\vartheta_{10}^2 - \vartheta_{11}^2)^2.
		\end{array}
	\end{equation}
	By taking the limit \(\tau(v) \to iI_4\ (v \to \Transpose{(1,-1,0,0)})\), we obtain the first four of \eqref{eq:quadratic-relations}.
	By adding them, we have the fifth equality.
	Regard the equalities between the left hand side and the intermediate side of them as a system of linear equations of unknown variables \(\vartheta_4^2\), \(\vartheta_5^2\), \(\vartheta_6^2\), \(\vartheta_7^2\).
	By solving it, we can express \(\vartheta_4^2\) and  \(\vartheta_6^2\) as linear combinations of \(\vartheta_0^2\), \(\vartheta_1^2\), \(\vartheta_2^2\), \(\vartheta_3^2\).
	Similarly, we have the expressions of \(\vartheta_8^2\) and \(\vartheta_{10}^2\).
\end{proof}

In order to show an analogue of the Jacobi formula, we consider a symplectic matrix so that it transforms the quadratic form \(\Transpose{v}Uv\) into the Lauricella hypergeometric series \(F_D\).
We define a matrix \(N\) as
\begin{equation}
	\label{eq:mat-N}
	N = \begin{pmatrix}
		N_{11} & N_{12} \\
		N_{21} & N_{22}
	\end{pmatrix},
\end{equation}
where
\(N_{11} =N_{22} =\diag(1,0,1,1)\), \(N_{21} =-N_{12}= \diag(0,1,0,0).\)
By computing the action of \(N\) on \(\tau = (\tau_{j,k}) \in \mathfrak{S}_4\), we obtain the following.
\begin{proposition}\label{prop:action of M0 1}
	We have  \(\chi(N,\tau)=\tau_{22}\)
	for \(\tau=(\tau_{jk})_{j,k} \in \mathfrak{S}_4\).
\end{proposition}
\begin{proof}
	The assertion follows from the definition of \(N\) and a direct computation.
\end{proof}
\begin{proposition}\label{prop:action of M0 2} We have
	\(\Transpose{v}Uv = 2i\,\tau(v)^\sharp_{22}\,v_1^2\) for \(v=\Transpose{
		(v_1,\dots,v_4)}\in \mathbb{B}_3\) and \(\tau(v)^\sharp = N\cdot \tau(v) = (\tau(v)_{jk}^\sharp)_{j,k}\).
\end{proposition}

\begin{proof}
	The action \(N \cdot \tau\) is given by
	\[
		\begin{pmatrix}
			\tau_{11} - \dfrac{\tau_{12}^2}{\tau_{22}}
			 & \dfrac{\tau_{12}}{\tau_{22}}
			 & -\dfrac{\tau_{12}\tau_{23}}{\tau_{22}} + \tau_{13}
			 & -\dfrac{\tau_{12}\tau_{24}}{\tau_{22}} + \tau_{14} \\[1.7ex]
			\dfrac{\tau_{12}}{\tau_{22}}
			 & -\dfrac{1}{\tau_{22}}
			 & \dfrac{\tau_{23}}{\tau_{22}}
			 & \dfrac{\tau_{24}}{\tau_{22}}                       \\[1.7ex]
			-\dfrac{\tau_{12}\tau_{23}}{\tau_{22}} + \tau_{13}
			 & \dfrac{\tau_{23}}{\tau_{22}}
			 & -\dfrac{\tau_{23}^2}{\tau_{22}} + \tau_{33}
			 & -\dfrac{\tau_{23}\tau_{24}}{\tau_{22}} + \tau_{34} \\[1.7ex]
			-\dfrac{\tau_{12}\tau_{24}}{\tau_{22}} + \tau_{14}
			 & \dfrac{\tau_{24}}{\tau_{22}}
			 & -\dfrac{\tau_{23}\tau_{24}}{\tau_{22}} + \tau_{34}
			 & -\dfrac{\tau_{24}^2}{\tau_{22}} + \tau_{44}
		\end{pmatrix}
	\]
	for a general element \(\tau \in \mathfrak{S}_4\).
	Moreover, by computing
	\[
		\tau(v)^\sharp = (N_{11}\tau(v)+N_{12})(N_{21}\tau(v)+N_{22})^{-1},
	\]
	we obtain the assertion.
\end{proof}

\begin{lemma}
	\label{lem:N action}
	We have
	\[
		\vt{a^\prime}{b^\prime}(N\cdot \tau) = \frac{1-i}{\sqrt{2}}\chi(N,\tau)^{1/2}\e\left( \phi_{a,b}(N) \right) \vt{a}{b}(\tau),
	\]
	where the branch of \(\chi(N,\tau)^{1/2}=(\tau_{2,2})^{1/2}\)
	is chosen so that its real part is positive, and
	\(\phi_{a,b}(N)\) is given
	in Lemma \ref{lem:Igusa-formula}.
\end{lemma}
\begin{theorem}\label{thm:Jacobi formula}
	We have
	\begin{equation}
		\vartheta_0(v)^2 + \vartheta_1(v)^2 = -\kappa^{1/2} \left( \Transpose{v}Uv \right) = -\frac{1}{16\pi \varGamma(3/4)^4}\left( \Transpose{v}Uv \right).\label{eq:Jacobi formula 1}
	\end{equation}
	Furthermore, the equality
	\begin{equation}\label{eq:Jacobi formula 3}
		\vartheta_{0000,0000}(\tau(v)^\sharp)^2 + \vartheta_{1100,0000}(\tau(v)^\sharp)^2 = \frac{\pi}{\varGamma(3/4)^4} F_D\left( \frac{1}{4},\frac{1}{4},\frac{1}{4},\frac{1}{4},1;x_1,x_2,x_3\right)^2
	\end{equation}
	holds, where \(\tau(v)^\sharp=N\cdot \tau(v)\) is the image of \(\tau(v)\) under the action of \(N\) in \eqref{eq:mat-N}.
\end{theorem}
\begin{proof}
	Note that \(\vartheta_0(v)^2+\vartheta_1(v)^2\) and \(-\Transpose{v}Uv\) are positive for
	\(v\in \mathcal{B}\cap \R^4\).
	The equality \eqref{eq:Jacobi formula 1} follows from Corollary \ref{cor:Thomae 3}.
	The left-hand side of \eqref{eq:Jacobi formula 3} becomes \( -i\tau(v)_{22}(\vartheta_0(v)^2 + \vartheta_1(v)^2)\) under the action of \(N\),
	since its automorphic factor
	under this action is
	\[
		\left( \tau(v)_{22}^{1/2} \exp(-\pi i/4) \right)^2 = -i\tau(v)_{22}.
	\]
	Furthermore, by \eqref{eq:Jacobi formula 1}, we have
	\[
		-i\tau(v)_{22}(\vartheta_0(v)^2 + \vartheta_1(v)^2) = i\kappa^{1/2}\tau(v)_{22}(\Transpose{v}Uv).
	\]
	By Proposition \ref{prop:action of M0 2} and
	\(\tau(v)^\sharp_{22} = -1/\tau(v)_{22}\), we have
	\[
		i(\Transpose{v}Uv) \tau(v)_{22} = 2v_1^2 = 16\pi^2 F_D\left( \frac{1}{4},\frac{1}{4},\frac{1}{4},\frac{1}{4},1;x_1,x_2,x_3 \right)^2,
	\]
	which yields the equality \eqref{eq:Jacobi formula 3}.
\end{proof}

\section{A Mean Generating Transformation}
\label{sec:2tau}
Recall that the map \(\mathbb{H}\ni \tau \mapsto 2\tau\in \mathbb{H}\) yields the arithmetic and geometric means of \(\vartheta_{0,0}(\tau)^2\) and \(\vartheta_{0,1}(\tau)^2\) as in \eqref{eq:2tau formula}.
We introduce a transformation of \(\mathbb{B}_3\),
which plays the role of an analogue of this map.

\begin{definition}[A mean generating transformation]
	\label{def:MGT-R}
	We define an element \(R\) in the unitary group
	\(\U(U,\Q(i))\) acting on \(\mathbb{B}_3\) by
	\[
		R = \frac{1}{1 - i} 		\begin{pmatrix}
			1 &   &    &    \\
			  & 2 &    &    \\
			  &   & 1  & -i \\
			  &   & -i & 1  \\
		\end{pmatrix},
	\]
	which is called a mean generating transformation.
	This matrix \(R\) factorizes into a product \(R = g_{1,3} R_1= R_1g_{1,3}\),
	where \(g_{1,3}\) is given in \eqref{eq:half-turn-Mat} and
	\[
		R_1 =\diag\left(\frac{1+i}{2},1+i,1,1\right).
	\]
\end{definition}

We show  in Section \ref{sec:AGM and Theta} that
it actually generates four means of four automorphic forms on
\(\mathbb{B}_3\) with respect to \(\Gamma\).

In this section, we determine the action of \(R\) on
theta constants \(\vt{a}{b}(v)\) on the complex ball \(\mathbb{B}_3\).
Recall that
we have studied the action of \(g_{1,3}\) in Proposition \ref{prop:g13 action}.

\subsection{The action of \(R_1\) on theta constants}
In order to determine the action of the matrix
\[
	\jmath(R_1) = \begin{pmatrix}
		1  & 0             & 0 & 0 & 0            & 1 & 0 & 0 \\
		0  & \tfrac{1}{2}  & 0 & 0 & \tfrac{1}{2} & 0 & 0 & 0 \\
		0  & 0             & 1 & 0 & 0            & 0 & 0 & 0 \\
		0  & 0             & 0 & 1 & 0            & 0 & 0 & 0 \\
		0  & -\tfrac{1}{2} & 0 & 0 & \tfrac{1}{2} & 0 & 0 & 0 \\
		-1 & 0             & 0 & 0 & 0            & 1 & 0 & 0 \\
		0  & 0             & 0 & 0 & 0            & 0 & 1 & 0 \\
		0  & 0             & 0 & 0 & 0            & 0 & 0 & 1
	\end{pmatrix} \in \Sp(8,\Q),
\]
we introduce its sub-matrix
\[
	S_1 =
	\begin{pmatrix}
		1  & 0            & 0           & 1 \\
		0  & \frac{1}{2}  & \frac{1}{2} & 0 \\
		0  & -\frac{1}{2} & \frac{1}{2} & 0 \\
		-1 & 0            & 0           & 1
	\end{pmatrix} \in \Sp(4,\Q)
\]
by selecting \(1,2,5,6\)-th rows and columns of \(\jmath(R_1)\).
We set a homomorphism
\[
	\Sp(4,\Q)\ni M=\begin{pmatrix} M_{11} & M_{12} \\ M_{21} & M_{22}
	\end{pmatrix}
	\mapsto M^\natural =
	\begin{pmatrix} M_{11}\oplus I_2 & M_{12}\oplus O_2 \\
                M_{21}\oplus O_2 & M_{22}\oplus I_2
	\end{pmatrix}\in \Sp(8,\Q),
\]
where \(M_{11}\oplus I_2=\begin{pmatrix} M_{11} & O_2 \\ O_2 &I_2
\end{pmatrix}\) and \(M_{12}\oplus O_2=\begin{pmatrix} M_{12} & O_2 \\ O_2 &O_2
\end{pmatrix}.\)
Note that
\[
	S_1^\natural=\jmath(R_1).
\]
We study the action of \(S_1\) on the theta constants
\(\vt{a}{b}(\tau)\) on \(\mathfrak{S}_2\).
We set the matrices
\[
	C_1 =
	\begin{pmatrix}
		0  & 0 & 1 & 0 \\
		0  & 1 & 0 & 0 \\
		-1 & 0 & 0 & 0 \\
		0  & 0 & 0 & 1
	\end{pmatrix},\ C_2=
	\begin{pmatrix}
		0 & 1 & 0 & 0 \\
		1 & 0 & 0 & 0 \\
		0 & 0 & 0 & 1 \\
		0 & 0 & 1 & 0
	\end{pmatrix} \in \Sp(4,\mathbb{Z}).
\]
Then, the matrix \(M_1 = C_2 C_1 S_1 C_1^{-1}\) is given by \(T\oplus T^{-1}\), where \(T\) is defined by
\[
	T=\frac{1}{2}
	\begin{pmatrix}
		1 & 1  \\
		1 & -1
	\end{pmatrix}.
\]
The matrix \(T\) satisfies \(\Transpose{T} = T\) and \(T^2 = \frac{1}{2}I_2\).
\begin{proposition}\label{prop:M1, inverse of B1, B1, B2 action}
	For \((a,b) \in \mathbb{Z}^2\times \mathbb{Z}^2\) and
	\(\tau_2 \in \mathfrak{S}_2\), the actions of \(M_1\), \(C_1\), \(C_1^{-1}\), and \(C_2\) on the theta constants \(\vartheta_{a,b}(\tau_2)\) are given as follows:
	\begin{enumerate}
		\item
		      \[
			      \vt{a}{b}(M_1\cdot \tau_2) = \vt{aT}{2bT}(\tau_2) + \vt{aT+e}{2bT}(\tau_2),\ e=(1,1);
		      \]
		\item for \((c,d) = C_1 \cdot (a,b) = (b_1,a_2,-a_1,b_2)\),
		      \[
			      \vt{a}{b}(C_1^{-1}\cdot \tau_2) = \frac{1-i}{\sqrt{2}} \exp\left( \frac{a_1 b_1}{2} \pi i \right) (\tau_{2;11})^{1/2} \vt{c}{d}(\tau_2),
		      \]
		      where \(\tau_{2;11}\) is the \((1,1)\)-entry of \(\tau_2\) and the argument of \(\tau_{2;11} \in \mathbb{H}\) is supposed to be \(0<\arg(\tau_{2;11}) < \pi\);
		\item for \((c,d) = C_1^{-1}\cdot(a,b) = (-b_1,a_2,a_1,b_2)\),
		      \[
			      \vt{a}{b}(C_1 \cdot \tau_2) = \frac{1+i}{\sqrt{2}} \exp\left( \frac{a_1b_1}{2}\pi i \right) (-\tau_{2;11})^{1/2} \vt{c}{d}(\tau_2),
		      \]
		      where the argument of \(-\tau_{2;11} \in -\mathbb{H}\) is supposed to be \(-\pi<\arg(-\tau_{2;11}) < 0\);
		\item for \((c,d) = C_2 \cdot (a,b) = (a_2,a_1,b_2,b_1)\),
		      \begin{equation}
			      \vt{a}{b}(C_2\cdot \tau_2) = \vt{c}{d}(\tau_2). \label{eq:B2 action}
		      \end{equation}
	\end{enumerate}
\end{proposition}
\begin{proof}
	We show the equality in (1).
	Note that \(M_1 \cdot \tau_2 = T \tau_2 T\),
	\(\mathbb{Z}^2 T /\mathbb{Z}^2 = \{[(1/2,1/2)],[(0,0)]\}\).
	The defining series of the left-hand side of (1) splits into
	the two series defining the theta constants in the right-hand side of (1).
	We can show the others by \eqref{eq:tr-theta} in Lemma \ref{lem:Igusa-formula}.
\end{proof}
We obtain the action of \(S_1\) by using the above formulas.
\begin{proposition}\label{prop:S1 action}
	Let \(\tau_2^{(1)}\) be the matrix \(C_2 M_1 C_1 \cdot \tau_2\)
	for \(\tau_2\in \mathfrak{S}_2\), and set
	\[
		(c,d) = S_1^{-1} \cdot (a,b) = \left( a_1-b_2,\frac{a_2-b_1}{2},\frac{a_2+b_1}{2},a_1+b_2 \right).
	\]
	Then, the theta constant \(\vt{a}{b}(S_1\cdot \tau_2)\) is equal to
	\[
		E(S_1)_{a,b}(\tau_{2;11})^{1/2}\left( \tau_{2;11}^{(1)} \right)^{1/2}\left(\vt{c}{d}(\tau_2)+\e\left( \frac{b_2-a_1}{4} \right)\vt{c+e^\prime_2}{d+e^\prime_1}(\tau_2) \right),
	\]
	where
	\[
		E(S_1)_{a,b} =-i \e\left( \frac{(a_2+b_1)(b_2-a_1)}{8} + \frac{a_1b_1}{4} \right),\ e^\prime_1 = (1,0),\ e^\prime_2 = (0,1),
	\]
	and \(\tau_{2;11}\) and \(\tau_{2;11}^{(1)}\) are the \((1,1)\)-entries of
	\(\tau_2\) and \(\tau_2^{(1)}\), respectively.
\end{proposition}
\begin{proof}
	Since \(M_1 = C_2 C_1 S_1 C_1^{-1}\),  \(S_1\) is
	equal to \( C_1^{-1}C_2 M_1 C_1\). Then we have
	\begin{align}
		\vartheta_{a,b}(C_1^{-1}C_2 M_1 C_1 \cdot \tau_2)
		& = \frac{1-i}{\sqrt{2}} \e \left( \frac{a_1b_1}{4} \right) \left( \tau_{2;11}^{(1)} \right)^{1/2} \vt{b_1,a_2}{-a_1,b_2}(C_2 M_1 C_1 \cdot \tau_2) \\
		& = \frac{1-i}{\sqrt{2}} \e \left( \frac{a_1b_1}{4} \right) \left( \tau_{2;11}^{(1)} \right)^{1/2} \vt{a_2,b_1}{b_2,-a_1}(M_1 C_1 \cdot \tau_2).
	\end{align}
	Moreover, we see that
	\begin{align}
		& \vt{a_2,b_1}{b_2,-a_1}(M_1 C_1 \cdot \tau_2)                                                                                                                               \\
		= & \vt{\frac{1}{2}(a_2+b_1),\frac{1}{2}(a_2-b_1)}{-a_1+b_2,a_1+b_2}(C_1\cdot \tau_2) + \vt{\frac{1}{2}(a_2+b_1)+1,\frac{1}{2}(a_2-b_1)+1}{-a_1+b_2,a_1+b_2}(C_1\cdot \tau_2), \\
		& \vt{\frac{1}{2}(a_2+b_1),\frac{1}{2}(a_2-b_1)}{-a_1+b_2,a_1+b_2}(C_1\cdot \tau_2)                                                                                          \\
		= & \frac{1+i}{\sqrt{2}} \e\left( -\frac{(a_1-b_2)(a_2+b_1)}{8} \right) (-\tau_{2;11})^{1/2} \vt{a_1-b_2,\frac{1}{2}(a_2-b_1)}{\frac{1}{2}(a_2+b_1),a_1+b_2}(\tau_2),          \\
		& \vt{\frac{1}{2}(a_2+b_1)+1,\frac{1}{2}(a_2-b_1)+1}{-a_1+b_2,a_1+b_2}(C_1\cdot \tau_2)                                                                                      \\
		= & \frac{1+i}{\sqrt{2}} \e\left( -\frac{(a_1-b_2)(a_2+b_1+2)}{8} \right) (-\tau_{2;11})^{1/2} \vt{a_1-b_2,\frac{1}{2}(a_2-b_1)+1}{\frac{1}{2}(a_2+b_1)+1,a_1+b_2}(\tau_2),
	\end{align}
	which yield the claim.
\end{proof}

By using Proposition \ref{prop:S1 action}, we can easily determine
the action of \(S_1^\natural \) on \(\vt{a}{b}(v)\).

\begin{corollary} \label{cor:R1 action}
	Set  \(\tau = \tau(v)\), \(\tau^{(1)} = (C_2M_1C_1)^\natural \cdot \tau\) and \((c,d) = \jmath(R_1)^{-1} \cdot (a,b)\) for \(v\in \mathbb{B}_3\). Then we have
	\[
		\vt{a}{b}(R_1 v) = E(R_1)_{a,b} \left( \tau_{11}^{(1)} \right)^{1/2} (\tau_{11})^{1/2}\left(\vt{c}{d}(v)\!+\!\e\left(\frac{b_2\!-\!a_1}{4} \right) \vt{c\!+\! e_2}{d\!+\! e_1}(v) \right),
	\]
	where \(\tau_{11}\) and \(\tau_{11}^{(1)}\) are the \((1,1)\)-entries of
	\(\tau\) and \(\tau^{(1)}\), respectively, and
	\[
		E(R_1)_{a,b} = -i\e \left( -\frac{(a_1-b_2)(a_2+b_1)}{8} + \frac{a_1b_1}{4} \right),\ e_1 = (1,0,0,0),\ e_2 = (0,1,0,0).
	\]
	Here, we choose the argument for each \(\tau_1 \in \mathbb{H}\) such that \(0 < \arg(\tau_1) < \pi\).
\end{corollary}

\subsection{The action of \(R\) on theta constants}
Since \(R=g_{1,3}R_1\),
the action of \(R\) can be determined by successive application of the obtained actions.

\begin{theorem}\label{thm:C action}
	We set \(\tau = \tau(v)\) \((v\in \mathbb{B}_3)\), \(\tau^{(1)} = (C_2M_1C_1)^\natural\cdot \tau\), and \((c,d) = (a,b)\jmath(R)\), and
	\begin{align}
		X^{(1)}_{a,b}(v) & =  \exp\left( -\frac{a_3+a_4}{2}\pi i \right) \vt{c}{d+e_3 + e_4}(v),                                  \\
		X^{(2)}_{a,b}(v) & = \vt{c+e_3 + e_4}{d}(v),                                                                              \\
		X^{(3)}_{a,b}(v) & = \exp\left( -\frac{a_1+a_3+a_4-b_2}{2}\pi i \right)\vt{c+e_2}{d+e_1 + e_3 + e_4}(v),                  \\
		X^{(4)}_{a,b}(v) & = \exp\left( -\frac{a_1-b_2}{2}\pi i \right) \vt{c+e_2 + e_3 + e_4}{d+e_1}(v),                         \\
		E(R)_{a,b}       & =\frac{1-i}{2}  \e \left( \frac{-a_1 a_2 + a_1 b_1 + a_2b_2 + b_1b_2 + (a_4-a_3)(b_4-b_3)}{8} \right),
	\end{align}
	where \(e_j\) is the \(j\)-th unit row vector of size \(4\) for \(j=1,2,3,4\).
	Then we have
	\[
		\vt{a}{b}(Rv) =E(R)_{a,b} (\tau_{11})^{1/2} (\tau^{(1)}_{11})^{1/2} \chi(\jmath(g_{1,3}), \imath(R_1 v))^{1/2} \sum_{j=1}^4 X^{(j)}_{a,b}(v),
	\]
	where the argument of \(\tau_{11}\) and \(\tau_{11}^{(1)}\) in \(\mathbb{H}\) are supposed to be
	\(
	0 < \arg(\tau_{11}), \arg(\tau_{11}^{(1)}) < \pi
	\)
	and the branch of \(\chi(\jmath(g_{1,3}), \imath(R_1 v))^{1/2}\)
	is assigned in Proposition \ref{prop:g13 action}.
	In particular, the following equalities hold:
	\begin{align}
		\vartheta_0(Rv) & = 2E_R(v)\left( \vartheta_8(v) + \vartheta_{10}(v) \right),       &
		\vartheta_1(Rv) & = 2E_R(v)\left( \vartheta_8(v) - \vartheta_{10}(v) \right),         \\
		\vartheta_4(Rv) & = \sqrt{2}\,E_R(v)\left( \vartheta_4(v) + \vartheta_6(v) \right), &
		\vartheta_5(Rv) & = \sqrt{2}\,E_R(v)\left( \vartheta_4(v) - \vartheta_6(v) \right),   \\
		\vartheta_8(Rv) & = E_R(v)\left( \vartheta_0(v) + \vartheta_2(v) \right),           &
		\vartheta_9(Rv) & = E_R(v)\left( \vartheta_0(v) - \vartheta_2(v) \right),
	\end{align}
	where
	\begin{equation}\label{eq:ER(v)}
		E_R(v) = \frac{1-i}{2}(\tau_{11})^{1/2} (\tau^{(1)}_{11})^{1/2} \chi(\jmath(g_{1,3}), \imath(R_1 v))^{1/2} .
	\end{equation}
\end{theorem}
\begin{proof}
	We have only to use Proposition \ref{prop:g13 action} and Corollary \ref{cor:R1 action}.
\end{proof}
\section{Main Result}\label{sec:main result}

\subsection{Expression of the AGM Through Theta Constants}\label{sec:AGM and Theta}
\begin{definition}\label{def:abc-functions}
	We define functions \(a,b_1,b_2,b_3\) on \(\mathbb{B}_3\) by
	\begin{align}
		a(v)   & = \vartheta_{0000,0000}(\tau(v)^\sharp)^2 +\vartheta_{1100,0000}(\tau(v)^\sharp)^2, \\
		b_1(v) & = \vartheta_{0000,0000}(\tau(v)^\sharp)^2 -\vartheta_{1100,0000}(\tau(v)^\sharp)^2, \\
		b_2(v) & = \vartheta_{0000,1100}(\tau(v)^\sharp)^2 +\vartheta_{1111,1111}(\tau(v)^\sharp)^2, \\
		b_3(v) & = \vartheta_{0000,1100}(\tau(v)^\sharp)^2 -\vartheta_{1111,1111}(\tau(v)^\sharp)^2,
	\end{align}
	where \(\tau(v)^\sharp =N \cdot \tau(v)\) for
	\(\tau(v)\in \mathfrak{S}_4\) and \(N\in \Sp(8,\Z)\) in \eqref{eq:mat-N}.
\end{definition}

\begin{lemma}\label{lam:1-xj}
	We express \(1-x_1\), \(1-x_2\), \(1-x_3\)
	in terms of  \(a(v)\), \(b_1(v)\), \(b_2(v)\), \(b_3(v)\)
	as
	\begin{align}
		1-x_1=\frac{b_1(v)^2}{a(v)^2}=\left(\frac{
			\vartheta_{0000,0000}(\tau(v)^\sharp)^2 -\vartheta_{1100,0000}(\tau(v)^\sharp)^2}
		{\vartheta_{0000,0000}(\tau(v)^\sharp)^2 +\vartheta_{1100,0000}(\tau(v)^\sharp)^2}
		\right)^2, \\
		1-x_2=\frac{b_2(v)^2}{a(v)^2}=\left(\frac{
			\vartheta_{0000,1100}(\tau(v)^\sharp)^2 +\vartheta_{1111,1111}(\tau(v)^\sharp)^2}
		{\vartheta_{0000,0000}(\tau(v)^\sharp)^2 +\vartheta_{1100,0000}(\tau(v)^\sharp)^2}
		\right)^2, \\
		1-x_3=\frac{b_3(v)^2}{a(v)^2}=\left(\frac{
			\vartheta_{0000,1100}(\tau(v)^\sharp)^2 -\vartheta_{1111,1111}(\tau(v)^\sharp)^2}
		{\vartheta_{0000,0000}(\tau(v)^\sharp)^2 +\vartheta_{1100,0000}(\tau(v)^\sharp)^2}
		\right)^2.
	\end{align}
\end{lemma}
\begin{proof}
	By applying Corollary \ref{cor:Thomae 3} to
	the expressions \(x_2,x_3\) in
	Corollary \ref{cor:x2, x3, (x3-x1)/(1-x1), (x2-x1)/(1-x1)},
	we have
	\[
		1-x_2=\left(\frac{\vartheta_2(v)^2-\vartheta_3(v)^2}
			{\vartheta_0(v)^2+\vartheta_1(v)^2}\right)^2
		,\quad
		1-x_3=\left(\frac{\vartheta_2(v)^2+\vartheta_3(v)^2}
			{\vartheta_0(v)^2+\vartheta_1(v)^2}\right)^2.
	\]
	Act \(N\) on \(\tau(v)\in \mathfrak{S}_4\)
	in these equalities and
	\[
		1-x_1=\left(\frac{\vartheta_0(v)^2-\vartheta_1(v)^2}
			{\vartheta_0(v)^2+\vartheta_1(v)^2}\right)^2
	\]
	obtained from the expression \(x_1\)
	in Proposition \ref{prop:x1,(x2-x3)/(1-x3)}.
	Here, note that
	\(N\cdot \nu_0\), \(N\cdot \nu_1\), \(N\cdot \nu_2\), \(N\cdot \nu_3\)
	are equivalent to
	\[
		(0000,0000),\quad  (1100,0000),\quad
		(0000,1100), \quad (1111,1111)
	\]
	modulo \(2\), respectively.
	Moreover, the formula \(\phi_{a,b}(N)\) in Lemma \ref{lem:Igusa-formula}
	yields \(\e\left( 2\phi_{a,b}(N) \right) = 1\)
	for \((a,b) = \nu_0,\nu_1,\nu_2\) and
	\(\e\left( 2\phi_{a,b}(N) \right) = -1\) for \((a,b) = \nu_3\).
	Hence the signs in the numerators of the expressions of
	\(1-x_2\) and \(1-x_3\) are changed under the action of \(N\).
\end{proof}


\begin{lemma} \label{lem:relation between automorphic factor}
	We have
	\begin{equation}\label{eq:chiN(tau(v))}
		8 E_R(v)^2 \chi(N,\imath(Rv)) = \chi(N,\imath(v)),
	\end{equation}
	where \(E_R(v)\) is given in \eqref{eq:ER(v)} in
	Theorem \ref{thm:C action}.
\end{lemma}
\begin{proof}
	The identity follows from a direct computation.
\end{proof}
\begin{theorem}\label{thm:analogy of 2tau formula}
	For \(v\in \mathbb{B}_3\), we have
	\[
		\begin{array}{ll}
			a(Rv)    = \dfrac{a(v)\!+\!b_1(v)\!+\!b_2(v)\!+\!b_3(v)}{4},    & b_1(R v)^2 = \dfrac{{(a(v)\!+\!b_3(v))(b_1(v)\!+\!b_2(v))}}{4}, \\
			b_2(R v)^2 = \dfrac{{(a(v)\!+\!b_2(v))(b_1(v)\!+\!b_3(v))}}{4}, & b_3(R v)^2 = \dfrac{{(a(v)\!+\!b_1(v))(b_2(v)\!+\!b_3(v))}}{4}.
		\end{array}
	\]
\end{theorem}
\begin{proof}
	By using Corollary \ref{cor:Thomae 3}, Lemmas \ref{lem:N action}, \ref{lem:relation between automorphic factor} and Theorem \ref{thm:C action}, we have
	\begin{align}
		a(Rv)
		& = -i \chi(N,\imath(Rv)) (\vartheta_0(Rv)^2+\vartheta_1(Rv)^2)                                 \\
		& = -i\chi(N,\imath(Rv))8E_R(v)^2 (\vartheta_8(v)^2 + \vartheta_{10}(v)^2)                      \\
		& = -i \chi(N,\imath(v)) (\vartheta_8(v)^2 + \vartheta_{10}(v)^2)
		= -\frac{1}{2}i \chi(N,\imath(v)) (\vartheta_0(v)^2 + \vartheta_2(v)^2)                          \\
		& = \frac{\vartheta_{0000,0000}(\tau(v)^\sharp)^2 + \vartheta_{0000,1100}(\tau(v)^\sharp)^2}{2}
		= \frac{a(v)+b_1(v)+b_2(v)+b_3(v)}{4}.
	\end{align}
	The function \(b_1(Rv)^2\) becomes
	\begin{align}
		& \left( -i \chi(N,\imath(Rv)) (\vartheta_0(Rv)^2-\vartheta_1(Rv)^2) \right)^2                                                                                                     \\
		= & \left( -i\chi(N,\imath(Rv))16E_R(v)^2 \vartheta_8(v) \vartheta_{10}(v) \right)^2
		= \left( -2i\chi(N,\imath(v)) \vartheta_8(v) \vartheta_{10}(v) \right)^2                                                                                                             \\
		= & \left( -2i\chi(N,\imath(v)) \frac{\vartheta_0(v)^2+\vartheta_1(v)^2+\vartheta_2(v)^2+\vartheta_3(v)^2}{4} \right)                                                                \\
		& \times \left( -2i\chi(N,\imath(v))\frac{\vartheta_0(v)^2-\vartheta_1(v)^2+\vartheta_2(v)^2-\vartheta_3(v)^2}{4} \right)                                                          \\
		= & \frac{\vartheta_{0000,0000}(\tau(v)^\sharp)^2+\vartheta_{1100,0000}(\tau(v)^\sharp)^2+\vartheta_{0000,1100}(\tau(v)^\sharp)^2-\vartheta_{1111,1111}(\tau(v)^\sharp)^2}{2}        \\
		& \times \frac{\vartheta_{0000,0000}(\tau(v)^\sharp)^2-\vartheta_{1100,0000}(\tau(v)^\sharp)^2+\vartheta_{0000,1100}(\tau(v)^\sharp)^2+\vartheta_{1111,1111}(\tau(v)^\sharp)^2}{2} \\
		= & \frac{(a(v)+b_3(v))(b_1(v)+b_2(v))}{4}.
	\end{align}
	Here, we use the expressions of \(\vartheta_8(v)^2\)
	and \(\vartheta_{10}(v)^2\) as linear combinations
	of \(\vartheta_0(v)^2\), \(\vartheta_1(v)^2\), \(\vartheta_2(v)^2\), \(
	\vartheta_3(v)^2\) in
	\eqref{eq:quadratic-relations} in Corollary \ref{cor:Thomae 3}.
	The same argument can be used to prove the others.
\end{proof}

\begin{corollary}\label{cor:analogy of 2tau formula}
	We define a subset \(\mathbb{B}_3^{123}\) in \(\mathbb{B}_3\) by
	\[
		\mathbb{B}_3^{123}=\{\per(x)\in \mathbb{B}_3\mid x \in X_{\mathbb{R}}^{123}\}.
	\]
	For \(v\in \mathbb{B}_3^{123}\), we have
	\begin{align}
		b_1(R v) & = \frac{\sqrt{(a(v)\!+\!b_3(v))(b_1(v)\!+\!b_2(v))}}{2}, \\
		b_2(R v) & = \frac{\sqrt{(a(v)\!+\!b_2(v))(b_1(v)\!+\!b_3(v))}}{2}, \\
		b_3(R v) & = \frac{\sqrt{(a(v)\!+\!b_1(v))(b_2(v)\!+\!b_3(v))}}{2}.
	\end{align}
\end{corollary}
\begin{proof}
	From Theorem \ref{thm:Thomae 1}, \(a(v)\), \(b_1(v)\), \(b_2(v)\) and \(b_3(v)\) never vanish on \(\mathbb{B}_3^{123}\).
	Since \(\mathbb{B}_3^{123}\) is simply  connected,
	the function \(\sqrt{1-x_j}a(v)/b_j(v)\) is a constant  \(\pm 1\)
	by Lemma \ref{lam:1-xj}, where we regard \(\sqrt{1-x_j}\) as a function
	\(y_j(v)\) in \(v\).
	We determine the sign by calculating the limit as \(v \to v_{56} = \Transpose{(0,1,0,0)}\). Since \(\lim_{v\to v_{56}}\limits y_j(v) = 1\), we obtain \(y_j(v) = b_j(v)/a(v)\) if \(\lim_{v\to v_{56}}\limits b_j(v)/a(v) = 1\).
	Indeed, we have
	\begin{align}
		\lim_{v\to v_{56}}\limits b_1(v)/a(v)
		& = \lim_{t \downarrow 0} \frac{b_1(v(t))}{a(v(t))}\quad (v(t) = \Transpose{(t,-1,0,0)})                                                                                                                                                              \\
		& = \lim_{t \downarrow 0} \frac{\vartheta_{0}(v(t))^2 - \vartheta_{1}(v(t))^2}{\vartheta_{0}(v(t))^2 + \vartheta_{1}(v(t))^2} = \lim_{t \downarrow 0}\frac{\vartheta_{00}(it)^4-\vartheta_{01}(it)^4}{\vartheta_{00}(it)^4+\vartheta_{01}(it)^4} = 1.
	\end{align}
	We can similarly prove the others.
\end{proof}

\begin{lemma}
	Each of \(a(R^n v),b_1(R^n v),b_2(R^n v),b_3(R^n v)\) converges to
	\(\vartheta_{00}(i)^4 = \pi/\varGamma(3/4)^4\) as \(n\to \infty\)
	for \(v \in \mathbb{B}_3^{123}\).
\end{lemma}
\begin{proof}
	Since \(R^{4n} = \diag(-1/4, -4, 1,1)^n\), \(R^{4n} v\) converges
	to \(v_{56} = \Transpose{(0,1,0,0)}\), which corresponds to
	\((x_1,x_2,x_3) = (0,0,0)\).
	Thus, by \eqref{eq:Jacobi formula 3} in Theorem \ref{thm:Jacobi formula}, we have
	\begin{align}
		\lim_{v \to v_{56}} a(v)
		& = \lim_{(x_1,x_2,x_3) \to (0,0,0)} \frac{\pi}{\varGamma(3/4)^4} F_D\left( \frac{1}{4},\frac{1}{4},\frac{1}{4},\frac{1}{4},1;x_1,x_2,x_3\right)^2 \\
		& = \frac{\pi}{\varGamma(3/4)^4} = \vartheta_{00}(i)^4.
	\end{align}
	Since \(y_j(v) = b_j(v)/a(v)\) and \(\lim_{v\to v_{56}}\limits y_j(v) = 1\)
	as in the proof of Corollary \ref{cor:analogy of 2tau formula},
	the assertions for \(b_1(v), b_2(v), b_3(v)\) also follow.
\end{proof}

\begin{theorem}
	\label{th:main}
	We take \(a \geq b \geq c \geq d > 0\), and set \(y_1 = b/a\), \(y_2 = c/a\), and \(y_3 = d/a\). Then the AGM \(M_\mathrm{KM}(a,b,c,d)\) is expressed in terms of Riemann's theta constants as
	\[
		\frac{a}{M_\mathrm{KM}(a,b,c,d)} = \frac{\varGamma(3/4)^4}{\pi} \left( \vartheta_{0000,0000}(\tau(v)^\sharp)^2 +\vartheta_{1100,0000}(\tau(v)^\sharp)^2  \right),
	\]
	where the period \(v\) is given by \(\per(1-y_1^2,1-y_2^2,1-y_3^2)\).
\end{theorem}
\begin{proof}
	Note that 	\(0\le 1-y_1^2\le 1-y_2^2\le 1-y_3^2<1\).
	From \cite[Remark 2]{KM09}, we have
	\begin{align}
		\frac{a}{M_\mathrm{KM}(a,b,c,d)} = \frac{1}{M_\mathrm{KM}(1,y_1,y_2,y_3)},
	\end{align}
	\begin{align}
		& M_\mathrm{KM}(1,y_1,y_2,y_3)
		= M_\mathrm{KM}(1,b_1(v)/a(v),b_2(v)/a(v),b_3(v)/a(v))                                                                                                                                                           \\
		= & \frac{1}{a(v)} M_\mathrm{KM}(a(v),b_1(v),b_2(v),b_3(v))
		= \frac{1}{a(v)} M_\mathrm{KM}\left( a(Rv),b_1(Rv),b_2(Rv),b_3(Rv) \right)                                                                                                                                       \\
		= & \cdots = \lim_{n\to \infty}\frac{1}{a(v)} M_\mathrm{KM}\left( a(R^n v),b_1(R^n v),b_2(R^n v),b_3(R^n v) \right)                                                                                              \\
		= & \frac{1}{a(v)} M_\mathrm{KM}\left( \frac{\pi}{\varGamma(3/4)^4},\frac{\pi}{\varGamma(3/4)^4},\frac{\pi}{\varGamma(3/4)^4},\frac{\pi}{\varGamma(3/4)^4} \right) = \frac{1}{a(v)}\frac{\pi}{\varGamma(3/4)^4}.
	\end{align}
	These yield the assertion.
\end{proof}

\begin{corollary}
	For \(0<d \leq c \leq b \leq a\), we have
	\begin{equation}		\frac{a}{M_{\mathrm{KM}}(a,b,c,d)}
		= F_D\left(\frac{1}{4} ,\frac{1}{4},\frac{1}{4},\frac{1}{4}, 1;1-\frac{b^2}{a^2},1-\frac{c^2}{a^2},1-\frac{d^2}{a^2}\right)^2.
	\end{equation}
\end{corollary}

\begin{proof}
	We have only to  use Theorems \ref{thm:Jacobi formula} and \ref{th:main}.
\end{proof}

\subsection{Borwein's Formula}
The period \(v=\per(x_1,x_2,x_3)\) becomes \(v_0 = \per(x_1,x_1,x_1) = \Transpose{(v_1,v_2,0,0)}\) when \(x_1=x_2=x_3\). Then we have
\(
\tau(v_0)^\sharp = \diag\left( -\dfrac{v_2}{v_1}i,-\dfrac{v_2}{v_1}i,i,i \right).
\)
We set \( \tau_1 = -\dfrac{v_2}{v_1}i \), and
\(\alpha(\tau_1)= \vartheta_{00}(\tau_1)^4 + \vartheta_{10}(\tau_1)^4\) and
\(\beta(\tau_1)= \vartheta_{00}(\tau_1)^4 - \vartheta_{10}(\tau_1)^4\),
which  respectively correspond to \(a\) and \(b\) in the notation of
\cite[Theorem 2.6]{Bor91}. Then we have
\begin{align}
	a(v_0)
	& =  \vartheta_{0000,0000}(\tau(v_0)^\sharp)^2 +
	\vartheta_{1100,0000}(\tau(v_0)^\sharp)^2                                                                                          \\
	& =  \left( \vartheta_{00}(\tau_1)^4 + \vartheta_{10}(\tau_1)^4 \right)\vartheta_{00}(i)^4  = \vartheta_{00}(i)^4 \alpha(\tau_1).\end{align}
We can similarly show \(b_1(v_0) = b_2(v_0) = b_3(v_0) = \vartheta_{00}(i)^4 \beta(\tau_1)\).
Therefore, when \(x_1=x_2=x_3\), \(a(v)\) reduces to \(\vartheta_{00}(i)^4\alpha(\tau_1)\), while \(b_1(v), b_2(v),\) and \(b_3(v)\) reduce to \(\vartheta_{00}(i)^4\beta(\tau_1)\).
Theorem \ref {th:main} yields the following.
\begin{corollary}
	For \(0 < b < a\), we define
	\[
		\tau_1 = \sqrt{2}i \left. F\left( \dfrac{1}{4},\dfrac{3}{4}, 1; \frac{b^2}{a^2}\right) \middle/ F\left( \dfrac{1}{4},\dfrac{3}{4}, 1; 1- \frac{b^2}{a^2}\right) \right..
	\]
	Then, the Borwein AGM is given by
	\[
		\frac{a}{M_\mathrm{Bor}(a,b)} = \vartheta_{00}(\tau_1)^4 + \vartheta_{10}(\tau_1)^4 = F\left( \frac{1}{4},\frac{3}{4}, 1; 1- \frac{b^2}{a^2}\right)^2.
	\]
\end{corollary}

\section*{Acknowledgments}
The authors are grateful to the referees
for the careful reading of the manuscript and
for many valuable comments and suggestions,
which help to improve this paper.

\printbibliography

\end{document}